\documentclass[12pt]{iopart}
 
\usepackage{iopams}

\usepackage{hyperref,graphicx}

\newcommand{\beq}{\begin{eqnarray}}
\newcommand{\eeq}{\end{eqnarray}}
\newcommand{\be}[1]{\begin{equation}\label{#1}}
\newcommand{\ee}{\end{equation}}

\newcommand{\eps}{{\eps}}
\newcommand{\N}{\mathbb N}

\newcommand{\R}{{\mathbb R}}

\newcommand{\Sp}{\mathbb S}
\renewcommand{\H}{\mathrm H}
\renewcommand{\L}{\mathrm L}
\def\1{\mathbb I}
\renewcommand{\(}{\left(}
\renewcommand{\)}{\right)}
\newcommand{\nrm}[2]{\|{#1}\|_{\L^{#2}(\R^d)}}
\newcommand{\nrmC}[2]{\|{#1}\|_{\L^{#2}(\mathcal C)}}
\newcommand{\nrmR}[2]{\|{#1}\|_{\L^{#2}(\R)}}
\newcommand{\inC}[1]{\int_{\mathcal C}#1\;dy}
\newcommand{\inR}[1]{\int_{\R}#1\;ds}
\newcommand{\ird}[1]{\int_{\R^d}#1\;dx}
\renewcommand{\eps}{\varepsilon}
\newcommand{\finproof}{\unskip\null\hfill$\square$\vskip 0.3cm}
\newenvironment{proof}{\par\smallskip\noindent\emph{Proof.} }{\finproof}
\newtheorem{theorem}{Theorem}

\newtheorem{lemma}[theorem]{Lemma}

\newtheorem{remark}[theorem]{Remark}
\newcommand{\nrmcnd}[2]{\|{#1}\|_{\L^{#2}(\mathcal C)}}

\renewcommand{\S}{{\Sp^{d-1}}}

\newcommand{\finprf}{\null\hfill$\square$\vskip 0.3cm}

\usepackage{color}
\definecolor{darkgreen}{rgb}{0,0.4,0}

\newcommand{\rmv}[1]{}

\newcommand{\eqref}[1]{(\ref{#1})}

\begin{document}
\title[Symmetry breaking in PDEs]{Branches of non-symmetric critical points and symmetry breaking in nonlinear elliptic partial differential equations}

\author{Jean Dolbeault, Maria J.~Esteban}

\address{Ceremade (UMR CNRS no.~7534), Univ.~Paris-Dauphine, Pl.~de Lattre de Tassigny, 75775 Paris Cedex~16, France}
\ead{dolbeaul@ceremade.dauphine.fr, esteban@ceremade.dauphine.fr}

\begin{abstract} In this paper we study the bifurcation of branches of non-symmetric solutions from the symmetric branch of solutions to the Euler-Lagrange equations satisfied by optimal functions in functional inequalities of Caffarelli-Kohn-Nirenberg type. We establish the asymptotic behavior of the branches for large values of the bifurcation parameter. We also perform an expansion in a neighborhood of the first bifurcation point on the branch of symmetric solutions, that characterizes the local behavior of the non-symmetric branch. These results are compatible with earlier numerical and theoretical observations. Further numerical results allow us to distinguish two global scenarios. This sheds a new light on the symmetry breaking phenomenon.\end{abstract}

\vspace*{-30pt}

\hspace*{1.68cm}\parbox{13.18cm}{\footnotesize\noindent\emph{Keywords}: Hardy-Sobolev inequality; Caffarelli-Kohn-Nirenberg inequality; extremal functions; Emden-Fowler transformation; radial symmetry; symmetry breaking; ground state; P\"oschl-Teller operator; bifurcation; elliptic equations; branches of solutions.}


\section{Introduction}\label{Sec:Intro}

In this paper we investigate how symmetry can be broken in some variational problems. Symmetry breaking occurs when antagonistic effects are competing, like weights or potentials (or coupling with other fields) on the one hand and nonlinearites on the other hand. An archetypal example for such issues is the question of symmetry of optimal functions in Caffarelli-Kohn-Nirenberg inequalities. While all terms are invariant under rotation around the origin, it is known that optimizers are not always radially symmetric. Caffarelli-Kohn-Nirenberg inequalities, also known as Hardy-Sobolev inequalities, is a particularly simple setting for the study of symmetry breaking because weights and nonlinear terms have simple homogeneity properties, so that Euler-Lagrange equations inherit scaling properties that allow to further simplify the study of the symmetry issues. Still, ranges of parameters for which optimizers are radially symmetric have not been completely determined yet.

Symmetry breaking issues are present in many areas of physics involving partial differential equations: quantum mechanics, mean field models, equations for phase transition, ferromagnetism, mechanics, \emph{etc.} Various mathematical methods are available either for proving symmetry (uniqueness, comparison techniques based for instance on moving plane methods, symmetrization: see for instance~\cite{Chou-Chu-93, MR1731336, springerlink:10.1007/s00526-011-0394-y}) or for proving symmetry breaking (multiplicity and bifurcation, energy, spectral methods). However, threshold cases are not characterized even in the simplest cases. 

A simple mechanism which can break symmetry is the instability of the \emph{symmetric extremals}, that is, the case where the extremals among radially symmetric functions are not local minima in the larger space of functions with no symmetry assumption. In the case of the Caffarelli-Kohn-Nirenberg inequalities, this instability has been studied in several papers (see~\cite{Felli-Schneider-03, Catrina-Wang-01,DDFT}) and the corresponding region of symmetry breaking is delimited by an explicit curve. However, it has been proved in~\cite{springerlink:10.1007/s00526-011-0394-y} that symmetry breaking can occur even in a range of parameters for which the \emph{symmetric extremals} are stable, that is, in cases where they are strict local minima. In order to understand this phenomenon, and symmetry breaking in general, we study the solution set of the Euler-Lagrange equations corresponding to a minimization problem associated with the Caffarelli-Kohn-Nirenberg inequalities. For those equations, we investigate the bifurcation of non-radially symmetric solutions from radially symmetric ones. The two theoretical contributions of the present paper are \emph{an asymptotic analysis of the branches for large values of the bifurcation parameter}, in Section~\ref{Sec:KnownAsymp}, and a detailed \emph{expansion of the non-radial solutions in a neighborhood of the bifurcation point} on the branch of radial extremals, in Section~\ref{Sec:Bifurcation}. Both results are consistent with known and new numerical results presented in Section \ref{Sec:NumericsAndScenarios} and give a significant insight into the local behaviour of the solutions, either around the bifurcation point or asymptotically. 

We shall consider a family of  Caffarelli-Kohn-Nirenberg inequalities which, for a given dimension $d\ge3$, depend on two exponents, $p\in(2,2^*]$ with $2^*:=2\,d/(d-2)$ and $\theta\in (0,1]$, and on a parameter $\Lambda>0$. 

For any dimension $d\ge 3$, let us consider the set $\mathcal D$ of all smooth functions which are compactly supported in $\R^d$. Define the numbers
\[
\vartheta(p,d):=d\,\frac{p-2}{2\,p}\,,\;\,a_c:=\frac{d-2}2\,,\; p(a,b):=\frac{2\,d}{d-2+2\,(b-a)}\,.
\]
 For any $a<a_c$, we consider the following \emph{Caffarelli-Kohn-Nirenberg inequalities}, which have been introduced in~\cite{Caffarelli-Kohn-Nirenberg-84} (also see~\cite{DDFT}):

\emph{Let $d\ge3$, $a<a_c$, $b\in[a,a+1]$ and assume that $p=p(a,b)$. Then, there exists a finite positive constant $\mathcal K_{\rm CKN}$ depending on $\theta$, $a$ and $p$ such that, for any $w\in\mathcal D$,
\be{CKNtheta}
\nrm{|x|^{-b}\,w}p^2\le\mathcal K_{\rm CKN}\,\nrm{|x|^{-a}\,\nabla w}2^{2\,\theta}\,\nrm{|x|^{-(a+1)}\,w}2^{2\,(1-\theta)}\,.
\ee}

According to~\cite{Catrina-Wang-01}, the Caffarelli-Kohn-Nirenberg inequalities on $\R^d$ can be rewritten in cylindrical variables using the Emden-Fowler transformation
\[
s=\log|x|\;,\quad\omega=\frac{x}{|x|}\in\S\,,\quad u(s,\omega)=|x|^{a_c-a}\,w(x)\,.
\]
The above inequalities are then equivalent to Gagliardo-Nirenberg-Sobolev inequalities on the cylinder $\mathcal C:=\R\times\S$ that can be written as
\be{CKNthetaC}
\nrmC up^2\leq\,\mathsf K_{\rm CKN}(\theta,\Lambda,p)\;\(\nrmC{\nabla u}2^2+\Lambda\,\nrmC u2^2\)^\theta\,\nrmC u2^{2\,(1-\theta)}\,,
\ee
for any $u\in H^1(\mathcal C)$, where $a$ and $\Lambda$ are related by $\Lambda=(a-a_c)^2$. \emph{Here we adopt the convention that the measure on $\S$ is the uniform probability measure}. Let us define
\be{Eqn:Q}
\mathcal Q_\Lambda^\theta[u]:=\frac{\(\nrmC{\nabla u}2^2+\Lambda\,\nrmC u2^2\)^\theta\,\nrmC u2^{2\,(1-\theta)}}{\nrmC up^2}\,.
\ee
In the case $\theta=1$, we shall simply write $\mathcal Q_\mu$ instead of $\mathcal Q_\mu^1$. We are interested in the map $\Lambda\mapsto\mathsf K_{\rm CKN}(\theta,\Lambda,p)$, what amounts to study the dependence of the minimum of $\mathcal Q_\Lambda^\theta$ on~$\Lambda$. The corresponding Euler-Lagrange equation is
\be{Euler-thetatheta}
-\,\theta\,\Delta u+\Big[(1-\theta)\,t[u]+\Lambda\Big] u=u^{p-1}\,,
\ee
with
\[
t[u]:=\frac{\inC{{|\nabla u|^2}}}{\inC{u^2}}\,.
\]
Let us introduce the parameter $\mu=\big((1-\theta)\,t[u]+\Lambda\big)/\theta$. Up to multiplication by a constant, the solutions of \eqref{Euler-thetatheta} are solutions of
\be{Euler-thetaequal1}
-\Delta u+\mu\,u=u^{p-1}\,.
\ee
If $\theta=1$, we may notice that $\Lambda=\mu$. Hence we may solve \eqref{Euler-thetaequal1}, denote by $u_\mu$ the corresponding solution which minimizes $\mathcal Q_\mu$, compute $\Lambda$ and then parametrize the solutions of \eqref{Euler-thetatheta} in terms of~$\mu$. Let us give some details. With
\[
\tau(\mu):=t[u_\mu]\quad\mbox{and}\quad\nu(\mu):=\frac{\nrmcnd{u_\mu}2^2}{\nrmcnd{u_\mu}p^2}\,,
\]
we can describe the set of solutions to \eqref{Euler-thetatheta} in parametric form as $\mu\mapsto(\Lambda^\theta(\mu),J^\theta(\mu))$ where\begin{eqnarray*}
&&\Lambda^\theta(\mu)=\theta\,\mu-(1-\theta)\,\tau(\mu)\,,\\
&&J^\theta(\mu):=\mathcal Q^\theta_\Lambda[u_\mu]=\nu(\mu)\,\theta^\theta\,(\mu+\tau(\mu))^\theta\,.
\end{eqnarray*}
The uniqueness of $u_\mu$ is not obvious and details will be provided in this paper.
We shall denote by $\tau_*(\mu)$, $\nu_*(\mu)$, $\Lambda_*^\theta(\mu)$ and $J_*^\theta(\mu)$ the corresponding quantities for the symmetric solution $u_{\mu,*}$ of \eqref{Euler-thetaequal1}.

With $\vartheta=\vartheta(p,d)=d\,\frac{p-2}{2\,p}$, denote by $\mathsf K_{\rm GN}=\mathsf K_{\rm GN}(p,d)$ the optimal constant in the Gagliardo-Nirenberg-Sobolev inequality
\be{Ineq:GNS}
\nrm up^2\le\frac{\mathsf K_{\rm GN}}{|\S|^\frac{p-2}p}\,\nrm{\nabla u}2^{2\,\vartheta}\nrm u2^{2\,(1-\vartheta)}\;,\quad\forall\,u\in\H^1(\R^d)\,.
\ee
Our first result is a direct consequence of the above parametrization and deals with the asymptotics of $\mathcal Q_\Lambda^\theta$ for large values of $\Lambda$.
\begin{theorem}\label{Thm:asymptott} With the previous notations, for all $\theta>\vartheta=\vartheta(p,d)$, we have
\[
\lim_{\mu\to\infty}\mu^{\vartheta-\theta}\,J^\theta(\mu)=\frac{\theta^\theta}{\vartheta^\vartheta}\,(1-\vartheta)^{\vartheta-\theta}\,\frac1{\mathsf K_{\rm GN}}\,.
\]
Moreover, the parametric curve $\mu\mapsto(\Lambda^\theta(\mu),J^\theta(\mu))$ is asymptotic to the curve
\[
\Lambda\mapsto\frac{\theta^\theta}{\vartheta(p,d)^{\vartheta(p,d)}\,(\theta-\vartheta(p,d))^{\theta-\vartheta(p,d)}}\,\frac{\Lambda^{\theta-\vartheta(p,d)}}{\mathsf K_{\rm GN}}\,,
\]
for large values of $\mu$ or, equivalently, for large values of $\Lambda=\Lambda^\theta(\mu)$.
\end{theorem}
The case $\theta=1$ has been established in~\cite[Theorem 1.2]{Catrina-Wang-01} and here we generalize it to the case $\theta<1$. The proof will be given in Section~\ref{Sec:KnownAsymp}.

\medskip Next we denote by $\mathsf K_{\rm CKN}^*(\theta,\Lambda,p)$ the optimal constant in \eqref{CKNthetaC} when the set of admissible functions is restricted to all symmetric functions in $\mathcal D$, \emph{i.e.} functions which depend only on $s$ and achieve their extremum at $s=0$. It is achieved by an explicit function $u_{\mu,*}$ with $\mu$ such that $\Lambda_*^\theta(\mu)=\Lambda$. We recall that  $\mathsf K_{\rm CKN}^*(\theta,\Lambda,p)$ is explicit (see \cite[Lemma~3]{DDFT}). Let us define
\be{Eqn:LambdaFS}
\Lambda_{\rm FS}(p,\theta):=4\,\frac{d-1}{p^2-4}\,\frac{(2\,\theta-1)\,p+2}{p+2}\,.
\ee
Symmetry breaking means $\mathsf K_{\rm CKN}^*(\theta,\Lambda,p)<\mathsf K_{\rm CKN}(\theta,\Lambda,p)$. It is known that
\begin{enumerate}
\item $\mathsf K_{\rm CKN}^*(\theta,\Lambda,p)=\mathsf K_{\rm CKN}(\theta,\Lambda,p)$ for $\Lambda>0$ small (see \cite{DEL2011}).
\item $\mathsf K_{\rm CKN}^*(\theta,\Lambda,p)<\mathsf K_{\rm CKN}(\theta,\Lambda,p)$ for $\Lambda>\Lambda_{\rm FS}(p,\theta)$  (see~\cite{Felli-Schneider-03, Catrina-Wang-01}).
\item The map $\Lambda\mapsto1/\mathsf K_{\rm CKN}(\theta,\Lambda,p)$ is increasing and concave when $\theta=1$.
\end{enumerate}
The value $\Lambda=\Lambda_{\rm FS}(p,\theta)$ corresponds to the threshold of instability of the symmetric minimizers of \eqref{Eqn:Q}. More estimates will be given in Section~\ref{Sec:KnownAsymp}. Our next purpose is to study the bifurcation of non-symmetric minimizers from the symmetric ones. Let us start with $\theta=1$ and define
\[
\mu_{\rm FS}:=\Lambda_{\rm FS}(p,1)\,.
\]
\begin{theorem}\label{Thm:T1} Assume that $\theta=1$, $d\ge3$ and $p\in (2, 2^* ]$. Under assumption \emph{(H)}, there exist a constant $c_{p,d}$ and
\be{ansatz-intro}
u_{(\mu)}:=u_{\mu,*}+\sqrt{c_{p,d}\,(\mu-\mu_{\rm FS})}\,\varphi+c_{p,d}\,(\mu-\mu_{\rm FS})\,\psi\,,
\ee
where $\varphi$ and $\psi$ are two smooth functions with exponential decay as $|s|\to\infty$ such that, for $c_{p,d}\,(\mu-\mu_{\rm FS})>0$, 
\[
\mathcal Q_\mu[u_{(\mu)}]=\mathcal Q_\mu[u_{\mu,*}]\(1-\frac{p^2-4}8\,c_{p,d}\,(\mu-\mu_{\rm FS})^2+o\((\mu-\mu_{\rm FS})^2\)\)\,.
\]
Moreover, if $c_{p,d}$ is positive, then for $\mu>\mu_{\rm FS}$, $\mathcal Q_\mu[u_{(\mu)}]$ minimizes $\mathcal Q_\mu$ in a neighborhood of~$u_{\mu,*}$ among smooth functions with exponential decay as $|s|\to\infty$, up to terms of order $o\((\mu-\mu_{\rm FS})^2\)$.
\end{theorem}
The assumption {\rm (H)} is rather technical but explicit and will be stated only in Section~\ref{Sec:Check}. For a given $d$, it is a condition on $p$, which ensures the existence on  $c_{p,d}$. Notice that the condition that $c_{p,d}$ is positive is stronger than {\rm (H)}. We are not able to fully characterize the positivity of $c_{p,d}$, but at least we will give a sufficient condition in Theorem~\ref{Thm:Check}. The corresponding range of $p$ is not expected to be optimal.

In Section~\ref{Sec:Bifurcation} we shall perform an expansion of the energy $\mathcal Q_\mu$ in a neighborhood of the first bifurcation  point on the symmetric branch, by minimizing $\mathcal Q_\mu$ among a special class of smooth functions with exponential decay, which is expected to contain all minimizers in $\H^1(\mathcal C)$. However we did not prove that such a regularity result holds order by order in the expansion. Anyway, our expansion provides us with an approximate, local minimizer under the condition that $c_{p,d}$ is positive.

The function $\varphi$ in Theorem \ref{Thm:T1} is explicit, $\psi$ is given by a linear elliptic equation with an explicit source term and $c_{p,d}$ is given by an identity involving $\varphi$ and $\psi$. Numerically, $c_{p,d}$ is positive in all cases considered in Section~\ref{Sec:NumericsAndScenarios}. Our last result is also written for $u_{(\mu)}$. With a slight abuse of notations, we may still use $\tau$ and $\nu$ for the approximated solution, that is
\[\mathrm{(A)}\hspace*{2cm}
\tau(\mu):=t[u_{(\mu)}]\quad\mbox{and}\quad\nu(\mu):=\frac{\nrmcnd{u_{(\mu)}}2^2}{\nrmcnd{u_{(\mu)}}p^2}\,,
\]
and then redefine $\Lambda^\theta(\mu)$ and $J^\theta(\mu)$ accordingly. Of course it is to be expected that the new versions of $\tau$ and $\nu$ differ from the former ones only by higher order terms but mathematically this is an open question. As we shall see in Section~\ref{Sec:NumericsAndScenarios} and with the above definition (A), $\tau'(\mu_{\rm FS})$ can be computed in terms of $c_{p,d}$. Let us define
\[
\vartheta_2(p,d):=\frac{\tau'(\mu_{\rm FS})}{1+\tau'(\mu_{\rm FS})}\,.
\]
One can expect that the value of $\vartheta_2(p,d)$ is the same if $\tau$ is computed  on the basis of $u_\mu$, the solution to the Euler-Lagrange equation~\eqref{Euler-thetaequal1}, or of $u_{(\mu)}$, the approximation defined by \eqref{ansatz-intro}. The relative values of $\vartheta_2(p,d)$ and $\vartheta(p,d)$ determine the behavior of the non-symmetric branch close to the bifurcation point on the symmetric branch. More precisely, we have the following local result.
\begin{theorem}\label{Thm:T2} Under the assumptions of Theorem~\ref{Thm:T1} and definition (A), if $c_{p,d}$ is positive, if $u_{(\mu)}$ is given by~\eqref{ansatz-intro} and if $\mu$ is taken in a right neighborhood of $\mu_{\rm FS}$, then we have the following alternative.\begin{description}
\item[$\bullet$] 
Either $\vartheta_2(p,d)\le\vartheta(p,d)$ and then for all $\theta\in(\vartheta(p,d),1]$, the branch $ (\Lambda^\theta(\mu), J^\theta(\mu))$ is concave, nondecreasing in $\mu$ and it is below the symmetric branch $(\Lambda_*^\theta(\mu), J_*^\theta(\mu))$. 
\item[$\bullet$] Or, on the contrary, $\vartheta_2(p,d)>\vartheta(p,d)$ and then we find two different behaviors:\begin{description}
\item[-] 
if $\theta\in(\vartheta_2(p,d),1]$, the branch is concave, nondecreasing in $\mu$ and below the symmetric branch,
\item[-] 
if $\theta\in(\vartheta(p,d),\vartheta_2(p,d))$, then the branch $(\Lambda^\theta(\mu),J^\theta(\mu))$ is above the symmetric branch $(\Lambda_*^\theta(\mu),J_*^\theta(\mu))$ and $\frac d{d\mu}\,\Lambda^\theta(\mu_{\rm FS})<0$.
\end{description}
\end{description}
\end{theorem}
In the last case, when $\theta\in(\vartheta(p,d),\vartheta_2(p,d))$, the branch in the $(\Lambda^\theta,J^\theta)$ representation is on the left of the bifurcation point and above the curve corresponding to symmetric solutions. The case $\theta=\vartheta(p,d)$ is of particular interest and will be discussed in details from a numerical point of view in Section~\ref{Sec:NumericsAndScenarios}.

\section{Preliminaries observations and proof of Theorem~\ref{Thm:asymptott}}\label{Sec:KnownAsymp}

\subsection{Caffarelli-Kohn-Nirenberg inequalities: more details on symmetry breaking}\label{Sec:CKN}

Recall that the threshold value for the stability of symmetric optimal functions is given by $\Lambda_{\rm FS}(p,\theta)$ defined in \eqref{Eqn:LambdaFS}: symmetry breaking occurs for any $\Lambda>\Lambda_{\rm FS}(p,\theta)$ according to~\cite{Felli-Schneider-03,DDFT} (also see~\cite{Catrina-Wang-01} for previous results and~\cite{MR2437030} if $d=2$ and $\theta=1$). As shown in~\cite{MR2437030,DELT09,springerlink:10.1007/s00526-011-0394-y}, there is a continuous curve $p\mapsto\Lambda_{\rm s}(p)$ with $\lim_{p\to2_+}\Lambda_{\rm s}(p)=\infty$ and $\Lambda_{\rm s}(p)>a_c^2$ for any $p\in (2,2^*)$ such that symmetry holds for any $\Lambda\le\Lambda_{\rm s}$ and there is symmetry breaking if $\Lambda>\Lambda_{\rm s}$. As proved in \cite{DEL2011}, for all $p,d$ in the considered range, $\theta=1$,
\[
\frac{(d-1)\,(6-p)}{4\,(p-2)} < \Lambda_{\rm s}(p) \le \Lambda_{\rm FS}(p,1)\,.
\]
Despite the fact that the exact shape of $\Lambda_s$ is not known, it can be proved that we have $\lim_{p\to2^*}\Lambda_{\rm s}(p)=a_c^2$ if $d\ge3$ and, if $d=2$, $\lim_{p\to\infty}\Lambda_{\rm s}(p)=0$, or more precisely, $\lim_{p\to\infty}p^2\Lambda_{\rm s}(p)=4$. Moreover, we also know from \cite[Theorem~3.1]{MR1734159} that symmetry holds if $\Lambda\le d^2/p^2$.

According to~\cite{1005}, it is known that an optimal function exists for any $\theta\in(\vartheta(p,d),1)$, but only if $\mathsf K_{\rm CKN}(\theta,\Lambda,p)>\mathsf K_{\rm GN}(p,d)$ when $\theta=\vartheta(p,d)$, where $\mathsf K_{\rm GN}(p,d)$ is the optimal constant in the  Gagliardo-Nirenberg-Sobolev inequality~\eqref{Ineq:GNS}. The case $\mathsf K_{\rm CKN}(\vartheta(p,d),\Lambda,p)=\mathsf K_{\rm GN}(p,d)$ has not been studied yet. A sufficient condition for the existence of extremals can be deduced, by comparison with symmetric functions, namely $\mathsf K_{\rm CKN}^*(\vartheta(p,d),\Lambda,p)>\mathsf K_{\rm GN}(p,d)$, which can be rephrased in terms of $\Lambda$ as $\Lambda<\Lambda_{\rm GN}^*(p,d)$ for some non-explicit (but easy to compute numerically) function $p\mapsto\Lambda_{\rm GN}^*(p,d)$. When $\theta=\vartheta(p,d)$ and $\Lambda>\Lambda_{\rm GN}^*(p,d)$, extremal functions (if they exist) cannot be symmetric and in the asymptotic regime $p\to2_+$, this condition is weaker than $\Lambda>\Lambda_{\rm FS}(p,\vartheta(p,d))$. One can indeed prove that $\lim_{p\to2_+}\Lambda_{\rm FS}(p,\vartheta(p,d))>\lim_{p\to2_+}\Lambda_{\rm GN}^*(p,d)$. Hence, for $\theta\in(\vartheta(p,d),1]$, close enough to $\vartheta(p,d)$ and $p-2>0$, small (but numerically not so small, actually, as shown in~\cite{Oslo}; also see \cite[Section 5]{springerlink:10.1007/s00526-011-0394-y} for estimates and Section~\ref{Sec:BifurcationsQualitativeDependence} for some plots), optimal functions exist and are not symmetric if $\Lambda>\Lambda_{\rm GN}^*(p,d)$, which is again a less restrictive condition than $\Lambda>\Lambda_{\rm FS}(p,\theta)$. See~\cite{springerlink:10.1007/s00526-011-0394-y} for proofs and~\cite{Oslo} for a more detailed overview of known results.

In this paper we study perturbatively the non-symmetric solutions lying in the first branch bifurcating from the branch of \emph{symmetric extremals} and show that they explain all phenomena of symmetry breaking known or observed so far, including cases in which the symmetric extremals are stable. Of course, it is not clear that all extremals for Caffarelli-Kohn-Nirenberg inequalities lie in those branches, even if probably that is the case. In Section~\ref{Sec:Bifurcation} we will provide a complete description of the branch around the bifurcation point, based on an asymptotic expansion. This clarifies the local behavior of the branch and accounts for all phenomena numerically observed in~\cite{Freefem}. Before doing so, let us study the branch of symmetric solutions and the asymptotic behavior of the branch of optimal functions (proof of Theorem~\ref{Thm:asymptott}).

\subsection{The case of symmetric extremals}\label{Sec:Prelim}

We start with the symmetric case for $\theta=1$ and adapt the computations that can be found in~\cite{DDFT} (also see~\cite{DEL2011} and the Appendix). Consider the equation
\be{Eqn:ODE}
-(p-2)^2\,w''+4\,w-2\,p\,|w|^{p-2}\,w=0\quad\mbox{in}\quad \R\,.
\ee
The function $w(s):=(\cosh s)^{-\frac2{p-2}}$ is, up to translations, the unique positive solution of~\eqref{Eqn:ODE}. As a consequence, the function $u(s)=\big(\frac12\,p\,\mu\big)^{1/(p-2)}\,w\(\frac{p-2}2\,\sqrt{\mu}\,s\)$ is the unique solution of
\be{eq-ustar}
-\,u''+\mu\,u=|u|^{p-2}\,u\quad\mbox{in}\quad \R\,.
\ee

 The symmetric optimal function $u_*$ for $\theta<1$ can be explicitly computed. Up to multiplication by a constant, $u_*$ solves
\[\label{eq-ustar-theta}
-\,\theta\,u_*''+\eta\,u_*=u_*^{p-1}\,,
\]
with $\eta=(1-\theta)\,t[u_*]+\Lambda$. After multiplying the above equation by $u_*$, integrating with respect to $s\in\R$ and dividing by $\int_{\R} u_*^2\,ds$, we find
\[
t[u_*]+\Lambda=\frac{\int_{\R} u_*^p\,ds}{\int_{\R} u_*^2\,ds}\,,
\]
where $u_*(s)=A\,w(B\,s)$, for all $s\in\R$, $w$ solves \eqref{Eqn:ODE}, $A=\(\frac{p\,\eta}2\)^\frac 1{p-2}$ and $B=\frac{p-2}2\,\sqrt{\frac\eta\theta}$. From this expression, as in~\cite{DEL2011}, we deduce that
\[
t[u_*]=B^2\,\frac{\mathsf J_2}{\mathsf I_2}=\frac{p-2}{p+2}\,\frac\eta\theta\quad\mbox{and}\quad\frac{\int_{\R} u_*^p\,ds}{\int_{\R} u_*^2\,ds}=A^{p-2}\,\frac{\mathsf I_p}{\mathsf I_2}=\frac{2\,p\,\eta}{p+2}\,,
\]
where for all $q\ge 2$, $\mathsf I_q:=\int_{\R}|w(s)|^q\,ds$, and $\mathsf J_2:=\int_{\R}|w'(s)|^2\;ds$ (see \ref{A1} for details). This provides the identity
\[
\frac{p-2}{p+2}\,\frac\eta\theta+\Lambda=\frac{2\,p\,\eta}{p+2}
\]
and uniquely determines $\eta=\frac{(p+2)\,\theta}{(2\,\theta-1)\,p+2}\,\Lambda$. As a consequence, we have
\[
t[u_*]=\frac{p-2}{(2\,\theta-1)\,p+2}\,\Lambda\,.
\]

\subsection{Gagliardo-Nirenberg inequalities and the corresponding asymptotic regime}\label{Sec:Gagliardo-Nirenberg}

Now we investigate the asymptotic regimes corresponding to $\Lambda\to\infty$ and prove Theorem~\ref{Thm:asymptott}. Let
\[
\mathsf S_p(\R^d):=\inf_{u\in\H^1(\R^d)\setminus\{0\}}\frac{\ird{|\nabla u|^2}+\ird{|u|^2}}{|\S|^\frac{p-2}p\,\(\ird{|u|^p}\)^\frac 2p}\,.
\]
An optimization of the quotient in the expression of $\mathsf S_p(\R^d)$ allows to relate this constant with $\mathsf K_{\rm GN}$. Indeed, if we optimize $\mathcal N[u]:=\ird{|\nabla u|^2}+\mu\ird{|u|^2}$ under the scaling $\lambda\mapsto u_\lambda(x):=\lambda^\frac dp\,u(\lambda\,x)$, we find that
\[
\mathcal N[u_\lambda]=\lambda^{2\,(1-\vartheta)}\ird{|\nabla u|^2}+\lambda^{-2\,\vartheta}\,\mu\ird{|u|^2}
\]
achieves its minimum at $\lambda_\star=\sqrt{\frac{\vartheta\,\mu}{1-\vartheta}}\,\frac{\nrm u2}{\nrm{\nabla u}2}$, so that
\[
\mathcal N[u_{\lambda_\star}]=\vartheta^{-\vartheta}\,(1-\vartheta)^{-(1-\vartheta)}\,\nrm{\nabla u}2^{2\,\vartheta}\nrm u2^{2\,(1-\vartheta)}\,\mu^{1-\vartheta}\,,
\]
thus proving that, with the choice $\mu=1$, $\mathsf K_{\rm GN}^{-1}=\vartheta^\vartheta\,(1-\vartheta)^{1-\vartheta}\,\mathsf S_p(\R^d)
$. For any $\mu>0$, if $u_\mu$ is the solution of \eqref{Euler-thetaequal1} and if it is a minimizer of $1/\mathsf K_{\rm CKN}(1,\Lambda,p)$, we know from~\cite[Theorem 1.2]{Catrina-Wang-01} that as $\mu\to\infty$,
\[
\big(\tau(\mu)+\mu\big)\,\nu(\mu)=\mathcal Q_\mu[u_\mu]\sim\mathsf S_p(\R^d)\,\mu^{1-\vartheta}\,.
\]
If $u$ is an optimal function for $\mathsf S_p(\R^d)$, we also know from the above computations that $\lambda_\star=1$, that is,
\[
1=\lambda_\star^2=\frac{\vartheta\,\mu}{1-\vartheta}\,\frac 1{\tau(\mu)}\quad \mbox{and so}\,,\quad \frac{\tau(\mu)}\mu=\frac\vartheta{1-\vartheta}\,.
\]
Hence,
\[
\nu(\mu)\sim(1-\vartheta)\,\mathsf S_p(\R^d)\,\mu^{-\vartheta}\quad\mbox{as}\quad\mu\to\infty\,.
\]

Consider now the case $\theta>\vartheta(p,d)$. According to the parametrization of Section~\ref{Sec:Intro}, that is, by definition of $J^\theta$ and $\Lambda^\theta$, we obtain that
\[
\Lambda^\theta(\mu)=\theta\,\mu-(1-\theta)\,\tau(\mu)=\frac{\theta-\vartheta(p,d)}{1-\vartheta(p,d)}\,\mu\,,
\]
\[
J^\theta(\mu)=\nu(\mu)\,\theta^\theta\,(\mu+\tau(\mu))^\theta\sim\theta^\theta\,(1-\vartheta(p,d))^{1-\theta}\,\mathsf S_p(\R^d)\,\mu^{\theta-\vartheta(p,d)}\,,
\]
as $\mu\to\infty$. Hence the parametric curve $\mu\mapsto(\Lambda^\theta(\mu),J^\theta(\mu))$ is asymptotic to the curve
\[
\Lambda\mapsto\frac{\theta^\theta\,(1-\vartheta(p,d))^{1-\vartheta(p,d)}}{(\theta-\vartheta(p,d))^{\theta-\vartheta(p,d)}}\,\mathsf S_p(\R^d)\,\Lambda^{\theta-\vartheta(p,d)}\,,
\]
for large values of $\mu$. This completes the proof of Theorem~\ref{Thm:asymptott}. See Figs.~\ref{fig1}--\ref{fig3} for some plots of the curves $\mu\mapsto(\Lambda^\theta(\mu),J^\theta(\mu))$ for various values of $\theta$ and how these curves can be compared with the ones corresponding to the asymptotic regime as described in Theorem~\ref{Thm:asymptott}.\finprf

The limit case $\theta=\vartheta=\vartheta(p,d)$ is of particular interest. Indeed, according to~\cite{1005}, Gagliardo-Nirenberg inequalities play a special role. See Fig.~\ref{fig6}. First of all, since $1/\mathsf K_{\rm CKN}(\vartheta,\Lambda^\vartheta(\mu),p)\le J^\vartheta(\mu)$ and using the fact that $\Lambda\mapsto\mathsf K_{\rm CKN}(\vartheta,\Lambda,p)$ is a non-increasing function of $\Lambda$, we recover the known result that
\[
\mathsf K_{\rm GN}\le\mathsf K_{\rm CKN}(\vartheta(p,d),\Lambda,p)\quad\forall\,\Lambda>0\,.
\]
Such an inequality has deep implications on the existence of an optimal function (see~\cite{1005} and in particular \cite[Theorem~1.4]{1005}): either the inequality is strict and there exists a non-trivial optimal function for~\eqref{CKNtheta}, or there is equality and a non-trivial optimal function may exist only if $\Lambda=\inf\{\lambda>0\,:\,\mathsf K_{\rm GN}\ge\mathsf K_{\rm CKN}(\vartheta(p,d),\lambda,p)\}$, but certainly not for any larger value of $\Lambda$, if the above infimum is finite.

In our setting, we can define $\mu_{\rm GN}:=\inf\{\mu>0\,:\,J^\vartheta(\mu)\le\mathsf K_{\rm GN}\}$, with $\vartheta=\vartheta(p,d)$. It is granted that $\mu_{\rm GN}>0$. Either $\mu_{\rm GN}=\infty$ and there is always a minimizer, or
\[
\mathsf K_{\rm CKN}(\vartheta(p,d),\Lambda^\vartheta(\mu),p)=J^\vartheta(\mu)\quad\forall\,\mu\in(0,\mu_{\rm GN}]
\]
and there exists a non-trivial optimal function for~\eqref{CKNtheta} if $\mu<\mu_{\rm GN}$, while
\[
\mathsf K_{\rm CKN}(\vartheta(p,d),\Lambda^\vartheta(\mu),p)=\mathsf K_{\rm GN}\quad\forall\, \mu\in[\mu_{\rm GN},\infty) \]
and there is no optimal function for~\eqref{CKNtheta} if $\mu>\mu_{\rm GN}$.

\section{An expansion at the bifurcation point: proof of Theorems \ref{Thm:T1} and \ref{Thm:T2}}\label{Sec:Bifurcation}

In this section, we determine the behavior of the branch of non-symmetric positive solutions that bifurcates from the branch of the symmetric ones in a neighborhood of the first bifurcation point $\mu=\mu_{\rm FS}$. Consider the case $\theta=1$ and denote by $u_{\mu,*}$ the positive symmetric solution~of
\be{eqmu}
-\Delta u+\mu\,u=u^{p-1}\,,
\ee
so that $\mathcal Q_\mu[u_{\mu,*}]=\nrmC{u_{\mu,*}}p^{p-2}=1/\mathsf K_{\rm CKN}^*(1,\mu,p)$. Notice that if $u$ is a solution to \eqref{eqmu}, we still have $\mathcal Q_\mu[u]=\nrmC up^{p-2}$ even if $u$ is not symmetric.

We will search for minimizers of $\mathcal Q_\mu$ in a restricted class of functions depending only on the variable $s$ (see Section~\ref{Sec:CKN}) along the axis of the cylinder and on the azimuthal angle $\zeta$ of the sphere because of the result on \emph{Schwarz foliated symmetry} of \cite{MR2499904}. This guarantees that we are in the right class for minimizers when $\theta=1$. For $\theta<1$, no such result has been established in the literature but we will work in the same framework. It is indeed straightforward to check that the same result holds.

Let $f_1$ be the first non-constant spherical harmonic function, \emph{i.e.}~the eigenfunction of the Laplace-Beltrami operator on the sphere $\S$ corresponding to the eigenvalue $d-1$ and denote by $f_2$ the next one (among the ones depending only on the azimuthal angle $\zeta$), with corresponding eigenvalue equal to $2\,d$. See~\ref{A4} for details. 

\subsection{Expansion of \texorpdfstring{$\mathcal Q_\mu$}{Qmu} at order two}\label{Sec:Bifurcation1}

Let us consider a solution of \eqref{eqmu} that can be written as
\[
u_\mu=u_{\mu,*}+\eps\,\varphi+o(\varepsilon)
\]
in a neighborhood of $u_{\mu,*}$. In the limiting regime corresponding to $\eps\to0$, an expansion at order two in $\eps$ gives
\[
\frac{\mathcal Q_\mu[u_{\mu,*}+\eps\,\varphi]}{\mathcal Q_\mu[u_{\mu,*}]}-1=\eps^2\,\frac{q[\mu,\varphi]}{\nrmC{u_{\mu,*}}p^p}+\eps^2\,(p-2)\,\frac{\(\inC{u_{\mu,*}^{p-1}\,\varphi}\)^2}{\nrmC{u_{\mu,*}}p^{2\,p}}+o(\eps^2)\,,
\]
where $q[\mu,\varphi]:=\inC{\(|\nabla\varphi|^2+\mu\,|\varphi|^2-(p-1)\,u_{\mu,*}^{p-2}\,|\varphi|^2\)}$. By minimizing the term of order two, we find that \[
\frac{\mathcal Q_\mu[u_{\mu,*}+\eps\,\varphi]}{\mathcal Q_\mu[u_{\mu,*}]}-1\sim\eps^2\,\frac{(\varphi_1,\mathcal H_\mu\,\varphi_1)_{L^2(\mathcal C)}}{\nrmC{u_{\mu,*}}p^p}\quad\mbox{as}\;\eps\to0\,,
\]
where $\mathcal H_\mu:=-\frac{d^2}{ds^2}+\mu+d-1-(p-1)\,u_{\mu,*}^{p-2}$ is a P\"oschl-Teller operator whose lowest eigenvalue is given by $\lambda_1(\mu)=d-1+\mu-\frac 14\,\mu\,p^2$, and such that $\varphi=\varphi_1\,f_1$ is the corresponding eigenfunction (see~\ref{A2} for details). Since $\varepsilon$ has not been specified yet, we can normalize $\varphi$ by the condition
\[
\nrmC\varphi2^2=\nrmR{\varphi_1}2^2=\nrmC{u_{\mu,*}}p^p\,,
\]
which slightly simplifies some computations below. This shows in particular that
\[
\frac{q[\mu,\varphi]}{\nrmC{u_{\mu,*}}p^p}=\lambda_1(\mu)\,.
\]
See \eqref{Eqn:varephi1} for an expression of $\varphi_1$, which is smooth and decays exponentially as $|s|\to\infty$.

As in~\cite{Felli-Schneider-03}, let $\mu_{\rm FS}$ be such that $\lambda_1(\mu_{\rm FS})=0$, that is
\[
\mu_{\rm FS}=4\,\frac{d-1}{p^2-4}
\]
(see~\ref{A2} for details). For any $\mu>\mu_{\rm FS}$, we have
\be{Eqn:muFS}
\lambda_1(\mu)=-\frac 14\,(p^2-4)\,(\mu-\mu_{\rm FS})\,.
\ee This determines the $O(\eps^2)$ term. Now we want to investigate the behavior of $\mathcal Q_\mu$ in a neighborhood of $\mu=\mu_{\rm FS}$ and therefore need an expansion at higher order.

\subsection{Expansion of \texorpdfstring{$\mathcal Q_\mu$}{Qmu} at order four}\label{Section:ExpansionEtaEps} 

Our purpose is to build an expansion $(u_\mu)_{\mu>\mu_{\rm FS}}$ of the branch of positive solutions of~\eqref{eqmu} that bifurcates from the branch $\mu\mapsto u_{\mu,*}$ at $u_{{\rm FS},*}=:u_{\rm FS}$ and satisfies $\mathcal Q_\mu[u_{(\mu)}]<\mathcal Q_\mu[u_{\mu,*}]$. For $\mu$ in a right-neighborhood of $\mu_{\rm FS}$, we look for solutions of~\eqref{eqmu} of the form $u_\mu=u_{(\mu)}$, up to higher order terms, where
\be{Eqn:Expansion}
u_{(\mu)}=u_{\mu,*}+\eps\,\varphi+\eta\,\psi\,,
\ee
with $\eps>0$ and $\eta=o(\eps)$. The fact that an expansion starting with the above expression can be built is standard. From now on, we will assume that the solutions are given by the above expression and that $\tau$, $\nu$, $\Lambda^\theta$ and $J^\theta$ are defined according to definition~(A). Some of our computations are formal, but can be justified by technical estimates that will be only sketched. Here $\varphi=\varphi_1\,f_1$ has been determined above. Recall that ~$\varphi_1$ is a function depending on~$s$ only. For convenience, let us define
\[
k_\psi:=\frac{\inC{u_{\mu,*}^{p-1}\,\psi}}{\inC{u_{\mu,*}^p}}\,.
\]
Since we are interested in functions depending only on the azimuthal angle $\zeta$, we indifferently use $\omega\in\S$ or $\zeta\in[0,\pi]$ with a slight abuse of notation. We consider the sequence $(f_k)_{k\in\N}$ of spherical harmonics depending only on $\zeta$. See~\hbox{\ref{A4}} for details. We denote by $\psi_k$ the decomposition of $\psi$ in spherical harmonics:
\[
\psi=k_\psi\,u_{\mu,*}+\sum_{k\ge0}\psi_k\,f_k\,,
\]
\[
\mbox{with}\quad\psi_k(r):=\int_\S\psi(r,\omega)\,f_k(\omega)\,d\nu(\omega)\quad\forall\,r\in[0,\infty)\,,
\]
where $d\nu(\omega)$ is the uniform probability measure on the sphere. Here we have chosen $\psi_0$ in such a way that $\inC{u_{\mu,*}^{p-1}\,\psi_0}=0$ because
\[
k_\psi=\frac{\inC{u_{\mu,*}^{p-1}\,(k_\psi\,u_{\mu,*}+\psi_0)}}{\inC{u_{\mu,*}^p}}=k_\psi+\frac{\inC{u_{\mu,*}^{p-1}\,\psi_0}}{\inC{u_{\mu,*}^p}}\,.
\]
We know that $q[\mu,\psi]=\sum_{k\ge0}q[\mu,\psi_k\,f_k]\,$, where
\[
q[\mu,\psi_k\,f_k]=\inC{|\nabla(\psi_k\,f_k)|^2}+\mu\inC{\psi_k^2}-(p-1)\inC{u_{\mu,*}^{p-2}\,\psi_k^2}\,.
\]
With $\mu_k:=\mu+k\,(k+d-2)\,,$
for any $k\ge 2$, we get that
\[
q[\mu,\psi_k\,f_k]=\inC{|\psi_k'|^2}+\mu_k\inC{\psi_k^2}-(p-1)\inC{u_{\mu,*}^{p-2}\,\psi_k^2}
\]
is nonnegative for $\mu-\mu_{\rm FS}>0$, small enough, and positive unless $\psi_k\equiv 0$. Lengthy but straightforward computations show that
\begin{eqnarray*}
\hspace{-1cm}\frac{\mathcal Q_\mu[u_{(\mu)}]}{\mathcal Q_\mu[u_{\mu,*}]}-1&=&a\,\eps^2+b\,\eps^4+c\,\eps^2\,\eta+d\,\eta^2+e\,\eps\,\eta
+o(\eps^4+\eta^2+|a|\,\eps^2+\eta\,\eps^2)
\end{eqnarray*}
when the function $\psi$ is smooth and has exponential decay as $|s|\to\infty$. The coefficients in the above expansion are given by
\[
a(\mu)=\frac{q[\mu,\varphi]}{\inC{u_{\mu,*}^p}}=\lambda_1(\mu)\,,
\]\rmv{
\begin{eqnarray*}
\frac{b(\mu)}{p-1}&=&-\frac{q[\mu,\varphi]\inC{u_{\mu,*}^{p-2}\,\varphi^2}}{\(\inC{u_{\mu,*}^p}\)^2}+\frac14\,(p-1)\,(p-2)\left[\frac{\inC{u_{\mu,*}^{p-2}\,\varphi^2}}{\inC{u_{\mu,*}^p}}\right]^2\\
&&-\frac1{12}\,(p-2)\,(p-3)\,\frac{\inC{u_{\mu,*}^{p-4}\,\varphi^4}}{\inC{u_{\mu,*}^p}}\,,
\end{eqnarray*}}
\begin{eqnarray*}
\frac{b(\mu)}{p-1}&=&-\,\lambda_1(\mu)\,\frac{\inC{u_{\mu,*}^{p-2}\,\varphi^2}}{\inC{u_{\mu,*}^p}}+\frac14\,(p-1)\,(p-2)\left[\frac{\inC{u_{\mu,*}^{p-2}\,\varphi^2}}{\inC{u_{\mu,*}^p}}\right]^2\\
&&-\frac1{12}\,(p-2)\,(p-3)\,\frac{\inC{u_{\mu,*}^{p-4}\,\varphi^4}}{\inC{u_{\mu,*}^p}}\,,
\end{eqnarray*}
\begin{eqnarray*}
c(\mu)&=&\,-2\,\lambda_1(\mu)\,k_\psi+(p-1)\,(p-2)\,\frac{\inC{u_{\mu,*}^{p-2}\,\varphi^2}}{\inC{u_{\mu,*}^p}}\,k_\psi\\
&&-(p-1)\,(p-2)\,\frac{\inC{u_{\mu,*}^{p-3}\,\varphi^2\,\psi}}{\inC{u_{\mu,*}^p}}\,,
\end{eqnarray*}
\[
d(\mu)=\frac{q[\mu,\psi]}{\inC{u_{\mu,*}^p}}+(p-2)\,k_\psi^2\,,
\]
\[
e(\mu)=2\,\lambda_1(\mu)\,\frac{\inC{\varphi\,\psi}}{\inC{u_{\mu,*}^p}}\,.
\]
With no restriction, we may require that $\varphi$ is optimal in the direction $f_1$, that is
\be{Constraint}
\inC{\varphi\,\psi}=0\,.
\ee
In other words, this amounts to require that $e(\mu)=0$ for any $\mu>\mu_{\rm FS}$.

According to \eqref{Eqn:muFS}, we get
\[
a(\mu)=-\frac 14\,(p^2-4)\,(\mu-\mu_{\rm FS})\,.
\]
Using $\int_0^\pi|f_1|^4\,d\nu=\frac{3\,d}{d+2}$ (see \ref{A4}), we obtain\rmv{
\begin{eqnarray*}
\frac{b(\mu)}{p-1}&=&-\lambda_1(\mu)\,\frac{\inR{u_{\mu,*}^{p-2}\,\varphi_1^2}}{\inR{u_{\mu,*}^p}}+\frac14\,(p-1)\,(p-2)\left[\frac{\inR{u_{\mu,*}^{p-2}\,\varphi_1^2}}{\inR{u_{\mu,*}^p}}\right]^2\\
&&-\frac{d\,(p-2)\,(p-3)}{4\,(d+2)}\,\frac{\inR{u_{\mu,*}^{p-4}\,\varphi_1^4}}{\inR{u_{\mu,*}^p}}\,.
\end{eqnarray*}}
\[
\frac{4\,b(\mu_{\rm FS})}{(p-1)\,(p-2)}=(p-1)\left[\frac{\inR{u_{\rm FS}^{p-2}\,\varphi_1^2}}{\inR{u_{\rm FS}^p}}\right]^2-\frac{d\,(p-3)}{d+2}\,\frac{\inR{u_{\rm FS}^{p-4}\,\varphi_1^4}}{\inR{u_{\rm FS}^p}}\,.
\]
All above integrals are computed in~\ref{A3} and allow to express $b(\mu_{\rm FS})$ as\rmv{
\begin{eqnarray*}
b(\mu)&=&\frac{(p-2)\,(p-1)\,p^2\,(p+2)}{4\,(3\,p-2)}\,\mu\,(\mu-\mu_{\rm FS})+\frac{p^4\,(p-1)^2\,(p-2)}{4\,(3\,p-2)^2}\,\mu^2\\
&&-\frac{d\,p^3\,(p-1)^2\,(p-2)\,(p-3)}{2\,(d+2)\,(3\,p-2)\,(5\,p-6)}\,\mu^2\,.
\end{eqnarray*}}
\[
b(\mu_{\rm FS})=\frac{4\,(d-1)^2\,p^3\,(p-1)^2\,\big[\,2\,p\,(5\,p-6)-d\,(p^2-16\,p+12)\,\big]}{(d+2)\,(p+2)^2\,(p-2)\,(3\,p-2)^2\,(5\,p-6)}\,.
\]
As for the terms which depend on $\eta$, we observe that they sum as
\begin{eqnarray*}
&&\hspace*{-12pt}\eps^2\,\eta\,c(\mu)+\eta^2\,d(\mu)+\eps\,\eta\,e(\mu)\\
&=&\eta^2\left[\frac{q[\mu,\psi]}{\nrmC{u_{\mu,*}}p^p}+(p-2)\,k_\psi^2\right]-2\,\lambda_1(\mu)\,k_\psi\,\eps^2\,\eta\\
&&+\eps^2\,\eta\,(p-1)\,(p-2)\left[\frac{\inC{u_{\mu,*}^{p-2}\,\varphi^2}}{\nrmC{u_{\mu,*}}p^p}\,k_\psi-\frac{\inC{u_{\mu,*}^{p-3}\,\varphi^2\,\psi}}{\nrmC{u_{\mu,*}}p^p}\right]\,.
\end{eqnarray*}
Using the fact that $f_1^2=f_0+\kappa_{(d)}\,f_2$ (see~\ref{A4}), it is straightforward to observe that the optimal function $\psi$ is given by
\be{Eqn:DecompPsi}
\psi=k_\psi\,u_{\mu,*}\,f_0+\psi_0\,f_0\,+\,\psi_1\,f_1\,+\,\psi_2\,f_2
\ee
while $\psi_k\equiv 0$ for any $k>2$ and solves the Euler-Lagrange equation
\begin{eqnarray*}
&-&\Delta\psi+\mu\,\psi-(p-1)\,u_{\mu,*}^{p-2}\,\psi+(p-2)\,k_\psi\,u_{\mu,*}^{p-1}\,f_0\\
&&+\frac{\eps^2}{2\,\eta}\,\Bigg[\(\frac{p^2\,(p-1)\,(p-2)}{3\,p-2}\,\mu-2\,\lambda_1(\mu)\)\,u_{\mu,*}^{p-1}\,f_0\\
&&\hspace*{36pt}-(p-1)\,(p-2)\,u_{\mu,*}^{p-3}\,\varphi_1^2\,(f_0+\kappa_{(d)}\,f_2)\Bigg]+\mathcal L\,\varphi_1\,f_1=0\,.
\end{eqnarray*}
Here we used the fact that $\inC{u_{\mu,*}^{p-2}\,\varphi^2}/\nrmC{u_{\mu,*}}p^p=\frac{p^2\,\mu}{3\,p-2}$ (see~\ref{A3} for details). Constraint~\eqref{Constraint} is taken into account through the Lagrange multiplier $\mathcal L$. Here higher order terms have been omitted: see Remark~\ref{Rem1}.

The three components $\psi_0$, $\psi_1$ and $\psi_2$ satisfy the equations
\begin{eqnarray*}
&&-\psi_0''+\mu\,\psi_0-\!\frac{p\,(p-1)\,\mu}{2\,[\cosh(\beta\,s)]^2}\,\psi_0\!+\!\frac{\beta^2\,A_0}{[\cosh(\beta\,s)]^{2\,\frac{p-1}{p-2}}}\!-\!\frac{\beta^2\,B_0}{[\cosh(\beta\,s)]^{2\,\frac{2\,p-3}{p-2}}}=0\,,\\
&&-\psi_1''+\mu_1\,\psi_1-\frac{p\,(p-1)\,\mu}{2\,[\cosh(\beta\,s)]^2}\,\psi_1+\mathcal L\,\varphi_1=0\,,\\
&&-\psi_2''+\mu_2\,\psi_2-\frac{p\,(p-1)\,\mu}{2\,[\cosh(\beta\,s)]^2}\,\psi_2-\frac{\beta^2\,B_2}{[\cosh(\beta\,s)]^{2\,\frac{2\,p-3}{p-2}}}=0\,,
\end{eqnarray*}
with $\mu_1=\mu+d-1$, $\mu_2=\mu+2\,d\;$ and
\begin{eqnarray*}
&&A_0=\frac{\alpha^{p-1}}{\beta^2}\,\frac{\eps^2}{2\,\eta}\(\frac{p^2\,(p-1)\,(p-2)}{3\,p-2}\,\mu-2\,\lambda_1(\mu)\)\,,\\
&&B_0=\frac{\eps^2}{2\,\eta}\,(p-1)\,(p-2)\,\frac{\alpha^{2\,p-3}}{\beta^2}\,,\\
&&B_2=\frac{\eps^2}{2\,\eta}\,(p-1)\,(p-2)\,\kappa_{(d)}\,\frac{\alpha^{2\,p-3}}{\beta^2}=\kappa_{(d)}\,B_0\,,
\end{eqnarray*}
Recall that $\beta^2=\frac 14\,(p-2)^2\,\mu$.

Multiplying the equation for $\psi_1$ by $\varphi_1$ and integrating by parts we get
\[ \lambda_1(\mu)\inR{\varphi_1\,\psi_1}+ \mathcal L \inR{|\varphi_1|^2}=0\,. \]
Using Assumption \eqref{Constraint}, this proves that $\mathcal L=0$. This implies that $\psi_1$ is an eigenfunction of $\mathcal H_\mu$, with eigenvalue $\lambda_1(\mu)$. Since $\lambda_1(\mu)$ is simple, we find that $\psi_1\equiv0$ by \eqref{Constraint}.

We may next observe that by taking
\be{Eqn:DecompPsi0}
\psi_0(s)=A_0\,\chi_{0,p-1}(\beta\,s)+B_0\,\chi_{0,2\,p-3}(\beta\,s)\;\mbox{and}\;\psi_2(s)=B_2\,\chi_{2,2\,p-3}(\beta\,s)\,,
\ee
the problem is reduced to the set of equations
\be{System}
\begin{array}{l}
-\chi_{0,p-1}''+\frac{4\,\chi_{0,p-1}}{(p-2)^2}-\frac{2\,p\,(p-1)\,\chi_{0,p-1}}{(p-2)^2\,(\cosh s)^2}+w^{p-1}=0\,,\cr
-\chi_{0,2\,p-3}''+\frac{4\,\chi_{0,2\,p-3}}{(p-2)^2}-\frac{2\,p\,(p-1)\,\chi_{0,2\,p-3}}{(p-2)^2\,(\cosh s)^2}-w^{2\,p-3}=0\,,\cr
-\chi_{2,2\,p-3}''+\frac{4\,\mu_2\,\chi_{2,2\,p-3}}{\mu\,(p-2)^2}-\frac{2\,p\,(p-1)\,\chi_{2,2\,p-3}}{(p-2)^2\,(\cosh s)^2}-w^{2\,p-3}=0\,,
\end{array}
\ee
where $w(s)=(\cosh s)^{-\frac2{p-2}}$. Since
\[
-\,w''+\frac{4\,w}{(p-2)^2}-\frac{2\,p\,(p-1)\,w}{(p-2)^2\,(\cosh s)^2}+\frac{2\,p}{p-2}\,w^{p-1}=0\,,
\]
it follows that
\[
\chi_{0,p-1}=\frac{p-2}{2\,p}\,w\,.
\]
We may notice that the equations for $\chi_{0,p-1}$, $\chi_{0,2\,p-3}$ and $\chi_{2,2\,p-3}$ are all independent of~$k_\psi$. Moreover, since $\inR{u_{\mu,*}^{p-1}\,\psi_0}=0$, we get
\[
\int_\R{w^{p-1}\(A_0\,\chi_{0,p-1}+B_0\,\chi_{0,2\,p-3}\)}\,ds=0\,,
\]
\emph{i.e.}
\[
0=A_0\inR{\frac{\chi_{0,p-1}}{(\cosh s)^{2\,\frac{p-1}{p-2}}}}+B_0\inR{\frac{\chi_{0,2\,p-3}}{(\cosh s)^{2\,\frac{p-1}{p-2}}}}\,.
\]
\begin{remark}\label{Rem1} The decomposition~\eqref{Eqn:DecompPsi} of $\psi$ is done up to higher order terms. Hence the above equality only holds for $\mu=\mu_{\rm FS}$, as can be checked by computing
\[
\inR{\chi_{0,p-1}\,w^{p-1}}=\frac{p-2}{2\,p}\inR{w^p}=\frac{p-2}{2\,p}\,\mathsf I_p
\]
and, using \eqref{System},
\[
\mathsf b_{0,p-1}:=\inR{\chi_{0,2\,p-3}\,w^{p-1}}=-\inR{\chi_{0,p-1}\,w^{2\,p-3}}=-\,\frac{p-2}{3\,p-2}\,\mathsf I_p\,.
\] (also see \ref{A1}). With $A_0=\frac{\alpha^{p-1}}{\beta^2}\,\frac{\eps^2}{2\,\eta}\,\(\frac{p^2\,(p-1)\,(p-2)}{3\,p-2}\,\mu-2\,\lambda_1(\mu)\)$ and $B_0=\frac{\eps^2}{2\,\eta}\,(p-1)\,(p-2)\,\frac{\alpha^{2\,p-3}}{\beta^2}$, we find that
\[
A_0\,\frac{p-2}{2\,p}\,\mathsf I_p-B_0\,\frac{p-2}{3\,p-2}\,\mathsf I_p=0
\]
holds if and only if $\lambda_1(\mu)=0$. As we shall see below, this is consistent with our expansion in terms of powers of $\eps$ and $\eta$ because for $\mu>\mu_{\rm FS}$, close enough to $\mu_{\rm FS}$, $\lambda_1(\mu)$ corresponds to a term of higher order.\end{remark}
The reader is invited to check that
\[
\chi_{0,2\,p-3}=-\,\frac 14\,\frac{p-2}{p-1}\(2\,w-w^{p-1}\)\,.
\]
As a consequence, one can compute
\be{b01}
\mathsf b_{0,1}:=\inR{\chi_{0,2\,p-3}\,w}=-\,\frac 14\,\frac{p-2}{p-1}\(2\,\mathsf I_2\!-\!\mathsf I_p\)=-\,\frac{p\,(p-2)}{2\,(p-1)\,(p+2)}\,\mathsf I_2\,.
\ee

Altogether we have found that
\[
\eps^2\,\eta\,c(\mu)+\eta^2\,d(\mu)+\eps\,\eta\,e(\mu)=\eta^2 \,Q[\psi]- \eps^2\,\eta \,L[\psi]
\]
up to higher order terms in $\eps$, $\eta$ and $(\mu-\mu_{\rm FS})$, where
\[
Q[\psi]:=\frac{q[\mu,\psi]}{\nrmC{u_{\mu,*}}p^p}+(p-2)\,k_\psi^2
\]
and
\[
L[\psi]:=-(p-1)\,(p-2)\left[\frac{\inC{u_{\mu,*}^{p-2}\,\varphi^2}}{\nrmC{u_{\mu,*}}p^p}\,k_\psi-\frac{\inC{u_{\mu,*}^{p-3}\,\varphi^2\,\psi}}{\nrmC{u_{\mu,*}}p^p}\right]
\]
are respectively quadratic and linear with respect to $\psi$. Since we can multiply $\psi$ by any positive constant $\nu$ and $\eta$ by $1/\nu$ simultaneously without changing the value of $ \eta^2 \,Q[\psi]- \eps^2\,\eta \,L[\psi]$, the optimal choice of $\eta$ in terms of $\varepsilon$ is
\be{Eqn:EtaEps}
\eta=\eps^2\,,
\ee
thus making the sum of the two terms equal to $\eps^4\(Q[\psi]-L[\psi]\)$ and $\psi$ independent of~$\eps$, at least at leading order. Moreover, if $\psi$ is a minimizer of $Q[\psi]-L[\psi]$, then it is straightforward to check that $2\,Q[\psi]-L[\psi]=0$, as follows by multiplying the Euler-Lagrange equation by $\psi$ and integrating. Altogether, we have found that
\[
\eps^2\,\eta\,c(\mu)+\eta^2\,d(\mu)+\eps\,\eta\,e(\mu)=-\frac 12\,\eps^4\,L[\psi]<0
\]
up to higher order terms in $\eps$ and for $\mu-\mu_{\rm FS}$ small enough, if $\psi$ is a minimizer of $Q[\psi]-L[\psi]$, that is,
\be{QL}
\frac{\mathcal Q_\mu[u_{(\mu)}]}{\mathcal Q_\mu[u_{\mu,*}]}-1-\lambda_1(\mu)\,\eps^2-b(\mu)\,\eps^4=-\frac 12\,L[\psi]\,\eps^4+o(\eps^4)\,.
\ee At this point, we may notice that the function $u_{(\mu)}$ has not been normalized. Multiplying it by a constant would not change the value of $\mathcal Q_\mu[u_{(\mu)}]$. If we want it to be a solution of~\eqref{eqmu} at leading order, then this implies that $\mathcal Q_\mu[u_{(\mu)}]=\nrmC{u_{(\mu)}}p^{p-2}$ and we may therefore impose the corresponding constraint, \emph{i.e.}
\[
\inC{|\nabla u_{(\mu)}|^2}+\mu\inC{u_{(\mu)}^2}=\inC{u_{(\mu)}^p}\,,
\]
without changing the equations written order by order (in other words, the Lagrange multiplier associated with this constraint is zero). Written in terms of $\varphi$ and $\psi$, at lowest order, that is at order $\eps^2$, this constraint amounts to
\[
\inC{\(|\nabla\varphi|^2+\mu\,|\varphi|^2-\frac{p\,(p-1)}2\,u_{\mu,*}^{p-2}\,|\varphi|^2\)}-\,(p-2)\,k_\psi\inC{u_{\mu,*}^p}=0\,.
\]
Hence by taking the limit as $\mu\to(\mu_{\rm FS})_+$ and observing that $\lambda_1(\mu)=O(\mu-\mu_{\rm FS})$, we find that, for $\mu=\mu_{\rm FS}$,
\[\label{kpsi}
k_\psi=-\,\frac12\,(p-1)\,\frac{\inC{u_{\mu,*}^{p-2}\,|\varphi|^2}}{\inC{u_{\mu,*}^p}}=-\, \frac{2\,p^2\,(p-1)\,(d-1)}{(p-2)\,(p+2)\,(3\,p-2)}\,.
\]
The explicit value of $k_\psi$ will however not be needed later, because of cancellations that occur in all subsequent computations.

Next comes the observation that, as long as we are interested in computing $L[\psi]$, we do not even need to normalize $u_{(\mu)}$ nor to compute $k_\psi$. Indeed with $\widetilde\psi:=\psi-k_\psi\,u_{\mu,*}=\psi_0\,f_0+\psi_2\,f_2$, we see that
\[
\frac{\inC{u_{\mu,*}^{p-2}\,\varphi^2}}{\nrmC{u_{\mu,*}}p^p}\,k_\psi-\frac{\inC{u_{\mu,*}^{p-3}\,\varphi^2\,\psi}}{\nrmC{u_{\mu,*}}p^p}=-\frac{\inC{u_{\mu,*}^{p-3}\,\varphi^2\,\widetilde\psi}}{\nrmC{u_{\mu,*}}p^p}\,,
\]
where $\widetilde\psi$ is fully determined by the coefficients $A_0$, $B_0$, and $B_2$, and by~\eqref{System}. This also determines
\[
L[\psi]=(p-1)\,(p-2)\,\frac{\inC{u_{\mu,*}^{p-3}\,\varphi^2\,\widetilde\psi}}{\nrmC{u_{\mu,*}}p^p}\,,
\]
which is not known explicitly but is independent of $k_\psi$ and can be computed as
\[
\frac{\inC{u_{\mu,*}^{p-3}\,\varphi^2\,\widetilde\psi}}{\nrmC{u_{\mu,*}}p^p}=\frac{A_0}\alpha\,\frac{p-2}{2\,p}\,\mathsf a_0+\frac{B_0}\alpha\,\alpha^{p-2}\,\frac{\mathsf b_{0,2\,p-3}}{\mathsf I_p}+\frac{B_2}\alpha\,\alpha^{p-2}\,\kappa_{(d)}\,\frac{\mathsf b_{2,2\,p-3}}{\mathsf I_p}\,,
\]
with
\[
\mathsf a_0:=\frac{\inC{u_{\mu,*}^{p-2}\,\varphi^2}}{\inC{u_{\mu,*}^p}}=\frac{p^2\,\mu}{3\,p-2}\,,\quad\mathsf b_{0,2\,p-3}:=\inR{\chi_{0,2\,p-3}\,w^{2\,p-3}}\,,
\]
\[
\hspace*{2cm}\mbox{and}\quad\mathsf b_{2,2\,p-3}:=\inR{\chi_{2,2\,p-3}\,w^{2\,p-3}}\,.
\]
Notice that here we have taken advantage of the facts that $\chi_{0,p-1}(\beta\,s)=\frac{p-2}{2\,p}\,\frac 1\alpha\,u(s)$ and $\varphi^2=\varphi_1^2\(f_0+\kappa_{(d)}\,f_2\)$. Using the expression of $\chi_{0,2\,p-3}$, we can also compute
\begin{eqnarray*}
\mathsf b_{0,2\,p-3}&=&-\,\frac 14\,\frac{p-2}{p-1}\(2\inR{w^{2\,(p-1)}}-\inR{w^{3\,p-4}}\)\\
&=&-\,\frac{(p-2)\,p\,(3\,p-4)}{(p-1)\,(3\,p-2)\,(5\,p-6)}\,\mathsf I_p\,.
\end{eqnarray*}
Altogether, with $\mathsf y:=\frac{\mathsf b_{2,2\,p-3}}{\mathsf I_p}$, we have found that
\be{Eqn:Lpsi}
L[\psi]=4\,(d-1)^2\,\frac{(p-1)\,p^3}{(p+2)^2}\!\left[\frac{(p-2)\,p}{(3\,p-2)^2\,(5\,p-6)}\!+2\,\frac{d-1}{d+2}\,\frac{p-1}{(p-2)^2}\,\mathsf y\right].
\ee

We have not been able to find an explicit expression for $\mathsf b_{2,2\,p-3}$, but we can prove that this is a positive quantity and even give an upper bound. According to \cite[p.~74]{Landau-Lifschitz-67}, the lowest eigenvalue of the P\"oschl-Teller operator $-\frac{d^2}{ds^2}-U_0\,(\cosh s)^{-2}$ is given~by 
\[
\lambda_0=\frac 12\,\sqrt{1+4\,U_0}-\frac 12-U_0\,,
\]
if we assume that $U_0$ is positive. Here we have that $U_0$ is given by
\[
U_0(p):=\frac{2\,p\,(p-1)}{(p-2)^2}
\]
and the reader is invited to check that
\begin{eqnarray*}
\sigma(p,d)&:=&\lambda_0+\frac{4\,\mu_2}{\mu_{\rm FS}\,(p-2)^2}\\
&=&\frac 12\,\sqrt{1+4\,U_0(p)}-\frac 12-U_0(p)+\frac2{(p-2)^2}\,\(2+\frac{d\,(p^2-4)}{d-1}\)
\end{eqnarray*}
is larger than $1$ for any $p>2$ and any $d\ge2$. As a straightforward consequence of \eqref{System}, we deduce that
\[
\sigma(p,d)\,\nrmR{\chi_{2,2\,p-3}}2^2\le\inR{\chi_{2,2\,p-3}\,w^{2\,p-3}}\le\nrmR{\chi_{2,2\,p-3}}2\,\nrmR{w^{2\,p-3}}2
\]
and, finally,
\[
\mathsf b_{2,2\,p-3}\le\frac1{\sigma(p,d)}\inR{w^{2\,(2\,p-3)}}=\frac{16\,p\,(p-1)\,(3\,p-4)}{(3\,p-2)\,(5\,p-6)\,(7\,p-10)}\,\frac{\mathsf I_p}{\sigma(p,d)}\,.
\]

\subsection{Optimization and a technical statement}\label{Sec:Qmu} 

Collecting the above estimates and using~\eqref{QL}, we get
\[
\frac{\mathcal Q_\mu[u_{(\mu)}]}{\mathcal Q_\mu[u_{\mu,*}]}=1-\frac 14\,(p^2-4)\,(\mu-\mu_{\rm FS})\,\eps^2+\left[b(\mu)-\frac 12\,L[\psi]\right]\eps^4+o(\eps^4)\,,
\]
provided $\mu-\mu_{\rm FS}=O(\eps^2)$. With $L[\psi]$ given by \eqref{Eqn:Lpsi}, assume that
\[\mathrm{(H)}\hspace*{2cm}
b(\mu_{\rm FS})-\frac 12\,L[\psi]\neq0\,.
\]
Let
\[
c_{p,d}:=\frac 18\,(p^2-4)\,{\left[b(\mu)-\frac 12\,L[\psi]\right]}_{|\mu=\mu_{\rm FS}}^{-1}\,.
\]
If $c_{p,d}$ is positive, in a neighborhood of $\mu=\mu_{\rm FS}$,
\[
\frac{\mathcal Q_\mu[u_{(\mu)}]}{\mathcal Q_\mu[u_{\mu,*}]}=1-\frac 14\,(p^2-4)\((\mu-\mu_{\rm FS})\,\eps^2-\frac{\eps^4}{2\,c_{p,d}}\)+o(\eps^4)
\]
is optimized, up to higher order terms, by taking
\be{positivityofcpd}
\eps^2=\eps^2(\mu)\sim c_{p,d}\,(\mu-\mu_{\rm FS})\quad\mbox{as}\quad\mu\to{\mu_{\rm FS}}_+\,.
\ee
\begin{remark} We may also consider the case $c_{p,d}<0$, which then requires that $\mu<\mu_{\rm FS}$. We will not emphasize it because we are interested in minimizers and $c_{p,d}<0$ means that we deal with local maximizers of $\mathcal Q_\mu$. Moreover, we have no example of such a case for specific values of $p$ and $d$.\end{remark}
Altogether, if  $c_{p,d}>0$, we have found that
\be{Eqn:ExpansionEnergy}
\frac{\mathcal Q_\mu[u_{(\mu)}]}{\mathcal Q_\mu[u_{\mu,*}]}=1-\frac{p^2-4}8\,c_{p,d}\,(\mu-\mu_{\rm FS})^2+o\((\mu-\mu_{\rm FS})^2\)
\ee
which ends the proof of Theorem \ref{Thm:T1}. A more detailed statement goes as follows.
\begin{theorem}\label{Thm:T1bis} In a neighborhood of $\mu=\mu_{\rm FS}$, if $u_{(\mu)}$ is given by \eqref{ansatz-intro},\begin{description}
\item[(i)] $\displaystyle\frac{\mathcal Q^1_\mu[u_{(\mu)}]}{\mathcal Q^1_\mu[u_{\mu,*}]}=1-\frac 14\,(p^2-4)\((\mu-\mu_{\rm FS})\,\eps^2-\frac{\eps^4}{2\,c_{p,d}}\)+o(\eps^4)\,,$
\item[(ii)] $\displaystyle\tau'(\mu_{\rm FS})=\frac{p-2}{p+2}+\,\frac{16\,p^2\,(d-1)^2}{(p-2)\,(p+2)^3}\,c_{p,d}\,,$
\item[(iii)] $\displaystyle\frac{\nu'(\mu_{\rm FS})}{\nu_*(\mu_{\rm FS})}=-\frac{p-2}{2\,p\,\mu_{\rm FS}}+c_{p,d}\left[\,p\,\mu_{\rm FS}\,\(\frac2{p+2}-\frac{p\,(p-1)}{3\,p-2}\)+\,2\,\frac{B_0}\alpha\,\(\frac{\mathsf b_{0,1}}{\mathsf I_2}+\frac{p-2}{3\,p-2}\)\right],$
\end{description}
where $c_{p,d}$, $B_0$ and $\mathsf b_{0,1}$ are explicit constants that have been defined above and that can be computed numerically. For \emph{(ii)} and \emph{(iii)}, we assume that $c_{p,d}$ is positive. 
\end{theorem}
Property (i) has already been established. Before proving (ii) and (iii), let us discuss the positivity of $c_{p,d}$.

\subsection{A sufficient condition for the positivity of \texorpdfstring{$c_{p,d}$}{cpd}}\label{Sec:Check} 

All above computations are valid under the assumption that $c_{p,d}$ \emph{is positive}, but this is not \emph{a priori} guaranteed. With the estimate of $\mathsf b_{2,2\,p-3}$ that has been found at the end of Section~\ref{Section:ExpansionEtaEps}, we can \emph{a posteriori} give a sufficient condition for the consistency of the method. Since
\begin{eqnarray*}
L[\psi]&\le&L_{\rm approx}[\psi]\\
&&\kern-4pt:=4\,(d-1)^2\,\frac{(p-1)\,p^3}{(p+2)^2}\left[\frac{(p-2)\,p}{(3\,p-2)^2\,(5\,p-6)}+2\,\frac{d-1}{d+2}\,\frac{p-1}{(p-2)^2}\,\mathsf y\right].
\end{eqnarray*}
with $\mathsf y:=\frac{16\,p\,(p-1)\,(3\,p-4)}{(3\,p-2)\,(5\,p-6)\,(7\,p-10)\,\sigma(p,d)}$, we know that $c_{p,d}$ is well defined and positive if $L_{\rm approx}[\psi]<2\,b(\mu_{\rm FS})$. Moreover, we have
\[
c_{p,d}\le c_{p,d}^{\rm approx}:=\frac 18\,(p^2-4)\,{\left[b(\mu)-\frac 12\,L_{\rm approx}[\psi]\right]}_{|\mu=\mu_{\rm FS}}^{-1}\,,
\]
at least as long as $L_{\rm approx}[\psi]\le2\,b(\mu_{\rm FS})$. Hence we have shown the following result.
\begin{theorem}\label{Thm:Check} The constant $c_{p,d}$ is positive if $p$ is contained in a non empty interval $(2,p_{\rm approx})\subset(2,2^*)$, where $p_{\rm approx}$ is defined as the largest root of the fourth order polynomial $p\mapsto\frac15\,(54-227\,d+103\,d^2)\,p^4-4\,(18-37\,d+25\,d^2)\,p^3+\frac83\,(63-67\,d+46\,d^2)\,p^2+16\,(d+3)\,(5\,d-3)\,p-240\,d\,(d+1)$. \end{theorem}
In practice, for all $d\ge 3$,  $ p_{\rm approx}(d)$ is close to $2\,d/(d-2)$ and converges to $2$ as $d\to\infty$. See Fig.~\ref{fig11} for an illustration. 

\subsection{Expansion of \texorpdfstring{$\tau(\mu)$}{taumu} around \texorpdfstring{$\mu_{\rm FS}$}{muFS}}\label{Sec:tau} 

Let $\tau(\mu)=t[u_{(\mu)}]=\frac{\inC{{|\nabla u_{(\mu)}|^2}}}{\inC{u_{(\mu)}^2}}$ and
\be{TauStar}
\tau_*(\mu):=t[u_{\mu,*}]=\frac{p-2}{p+2}\,\mu\,.
\ee
We can notice that $\tau(\mu_{\rm FS})=\beta^2\,\mathsf J_2/\mathsf I_2$ (see~\ref{A1} and~\ref{A2}) so that
\[
\tau(\mu_{\rm FS})=\frac{p-2}{p+2}\,\mu_{\rm FS}=\frac{4\,(d-1)}{(p+2)^2}\,.
\]
With the above expressions in hand, we can now compute the derivative
\[
\tau'(\mu_{\rm FS})=\frac d{d\mu}\,t[u_{(\mu)}]_{|\mu=\mu_{\rm FS}}\,.
\]
{}From~\eqref{TauStar} we know that $t[u_{\mu,*}]-t[u_{\rm FS}]=\frac{p-2}{p+2}\,(\mu-\mu_{\rm FS})$. By expanding the expression $\tau(\mu)-\tau(\mu_{\rm FS})=t[u_{(\mu)}]-t[u_{\rm FS}]=t[u_{\mu,*}]-t[u_{\rm FS}]+t[u_{(\mu)}]-t[u_{\mu,*}]$ in powers of~$\eps$ with $u_ {(\mu)}=u_{\mu,*}+\eps\,\varphi_1\,f_1+\eps^2\psi+o(\eps^2)$ where $\varphi$, $\psi$ have been chosen in Section \ref{Section:ExpansionEtaEps}, we get, up to higher order terms,
\begin{eqnarray*}
t[u_{(\mu)}]\!-\!t[u_{\mu,*}]&\sim&\eps^2\left[(\lambda_1(\mu)\!-\!\mu\!-\!t[u_{\mu,*}])\,\frac{\inC{\varphi^2}}{\inC{u_{\mu,*}^2}}+(p\!-\!1)\,\frac{\inC{u_{\mu,*}^{p-2}\,\varphi^2}}{\inC{u_{\mu,*}^2}}\right.\\
&&\hspace*{1.3cm}\left.-\,2\,(\mu+t[u_{\mu,*}])\,\frac{\inC{u_{\mu,*}\,\psi}}{\inC{u_{\mu,*}^2}}+{2\, \frac{\inC{u_{\mu,*}^{p-1}\,\psi}}{\inC{u_{\mu,*}^2}}}\right]\,,
\end{eqnarray*}
and, by computing as above, we find that
\begin{eqnarray*}\label{tautest}
\tau'(\mu_{\rm FS})&=&\frac{p-2}{p+2}+c_{p,d}\left[\,-\,\frac{2\,p\,\mu_{\rm FS}}{p+2}\frac{\inC{\varphi_1^2}}{\inC{u_{\rm FS}^2}}+(p-1)\,\frac{\inC{u_{\rm FS}^{p-2}\,\varphi_1^2}}{\inC{u_{\rm FS}^2}}\right.
\\
&&\hspace*{3cm}-\,\left.\frac{4\,p\,\mu_{\rm FS}}{p+2}\,\frac{\inC{u_{\rm FS}\,\psi_0}}{\inC{u_{\rm FS}^2}}\right]\,,\nonumber
\end{eqnarray*}
because we notice that the terms involving $k_\psi$ cancel. Hence, using~\eqref{Eqn:DecompPsi} and \eqref{Eqn:DecompPsi0}, we have found that
\begin{eqnarray*}
\tau'(\mu_{\rm FS})&=&\frac{p-2}{p+2}+c_{p,d}\left[\,-\,\frac{4\,p^2\,\mu_{\rm FS}^2}{(p+2)^2} + \frac{2\,(p-1)\,p^3\,\mu_{\rm FS}^2}{(p+2)\,(3\,p-2)}\right.\\
&&\hspace*{3cm}\left.-\,\frac{2\,(p-2)}{p+2}\,\mu_{\rm FS}\,\frac{A_0}\alpha-\frac{4\,p\,\mu_{\rm FS}}{p+2}\frac{B_0}\alpha\,\frac{\mathsf b_{0,1}}{\mathsf I_2}\right]\\
&&=\frac{p-2}{p+2}-\,c_{p,d}\left[\frac{8\,p\,(d-1)}{(p-2)\,(p+2)^2}\right]^2\left[1+\frac{(p-1)\,p\,(p+2)}{2\,(p-2)}\frac{\mathsf b_{0,1}}{\mathsf I_2}\right].
\end{eqnarray*}
Here the coefficient $\mathsf b_{0,1}$ is given by $\mathsf b_{0,1}:=\inR{\chi_{0,2\,p-3}\,w}$. Using~\eqref{b01}, this proves part~(ii) of Theorem~\ref{Thm:T1bis}.

\subsection{Expansion of \texorpdfstring{$\nu(\mu)$}{numu} around \texorpdfstring{$\mu_{\rm FS}$}{muFS}}\label{Sec:nu} 

Let us consider $\nu(\mu):=\nrmcnd{u_{(\mu)}}2^2/\nrmcnd{u_{(\mu)}}p^2$. Again we can write
\[
\nu(\mu)-\nu(\mu_{\rm FS})=\(\frac{\nrmcnd{u_{\mu,*}}2^2}{\nrmcnd{u_{\mu,*}}p^2}-\frac{\nrmcnd{u_{\rm FS}}2^2}{\nrmcnd{u_{\rm FS}}p^2}\)+\(\nu(\mu)-\frac{\nrmcnd{u_{\mu,*}}2^2}{\nrmcnd{u_{\mu,*}}p^2}\)\,.
\]
With $\beta=\frac{p-2}2\,\sqrt\mu$, using expressions that can be found in~\ref{A1}, we see that
\be{NuStar}
\nu_*(\mu):=\frac{\nrmcnd{u_{\mu,*}}2^2}{\nrmcnd{u_{\mu,*}}p^2}=\beta^{-\frac{p-2}p}\,\frac{\mathsf I_2}{\mathsf I_p^{2/p}}=\kappa_p\,\mu^{-\frac{p-2}{2\,p}}\,,
\ee
with $\kappa_p=\(\frac{p+2}4\)^\frac 2p\(\frac{2\,\sqrt\pi}{p-2}\,\frac{\Gamma\(\frac2{p-2}\)}{\Gamma\(\frac2{p-2}+\frac 12\)}\)^\frac{p-2}p$, and hence
\[
\nu_*'(\mu_{\rm FS})=\lim_{\mu\to\mu_{\rm FS}}\frac1{\mu-\mu_{\rm FS}}\(\frac{\nrmcnd{u_{\mu,*}}2^2}{\nrmcnd{u_{\mu,*}}p^2}-\frac{\nrmcnd{u_{\rm FS}}2^2}{\nrmcnd{u_{\rm FS}}p^2}\)=-\frac{p-2}{2\,p}\,\frac{\nu_*(\mu_{\rm FS})}{\mu_{\rm FS}}\,.
\]
If $u_{(\mu)}=u_{\mu,*}+\eps\,\varphi+\eps^2\,\psi$, where $\varphi$, $\psi$ have been chosen in Section \ref{Section:ExpansionEtaEps}, after a Taylor expansion we find that
\begin{eqnarray*}\label{nutest}
\nu(\mu)&=&\nu_*(\mu)\left[1+\eps^2\(\frac{\inC{\varphi^2}}{\inC{u_{\mu,*}^2}}-(p-1)\,\frac{\inC{u_{\mu,*}^{p-2}\,\varphi^2}}{\inC{u_{\mu,*}^p}}\)\right.\\
&&\qquad\qquad\left.+\,2\,\eps^2\,\frac{\inC{u_{\mu,*}\,\psi}}{\inC{u_{\mu,*}^2}}{-\,
2\,\eps^2\,\frac{\inC{u_{\mu,*}^{p-1}\,\psi}}{\inC{u_{\mu,*}^p}}}+o(\eps^2)\right]\,.\nonumber
\end{eqnarray*}
Again we may notice that the terms involving $k_\psi$ cancel and, based on~\eqref{Eqn:DecompPsi} and \eqref{Eqn:DecompPsi0}, we arrive at
\begin{eqnarray*}
\frac{\nu'(\mu_{\rm FS})}{\nu_*(\mu_{\rm FS})}&=&-\frac{p-2}{2\,p\,\mu_{\rm FS}}+c_{p,d}\left[\,p\,\mu_{\rm FS}\,\(\frac2{p+2}-\frac{p\,(p-1)}{3\,p-2}\)\right.
\\
&&\hspace*{4cm}\left.+\,2\,\(\frac{B_0}\alpha\,\frac{\mathsf b_{0,1}}{\mathsf I_2}-\frac{B_0}\alpha\,\frac{\mathsf b_{0,p-1}}{\mathsf I_p}\)\right]\\
&&\qquad=-\frac{p-2}{2\,p\,\mu_{\rm FS}}+c_{p,d}\left[\,p\,\mu_{\rm FS}\,\(\frac2{p+2}-\frac{p\,(p-1)}{3\,p-2}\)\right.
\\
&&\hspace*{45mm}\left.+\,2\,\frac{B_0}\alpha\,\(\frac{\mathsf b_{0,1}}{\mathsf I_2}+\frac{p-2}{3\,p-2}\)\right]\,.
\end{eqnarray*}
\begin{lemma}\label{Rem2} At the bifurcation point $\mu=\mu_{\rm FS}$ we get the following.
\[\label{Identity:ThetaUnStar}
\frac{\nu_*'(\mu_{\rm FS})}{\nu_*(\mu_{\rm FS})}+\frac{\tau_*'(\mu_{\rm FS})}{\mu_{\rm FS}+\tau_*(\mu_{\rm FS})}=\frac{\nu'(\mu_{\rm FS})}{\nu_*(\mu_{\rm FS})}+\frac{\tau'(\mu_{\rm FS})}{\mu_{\rm FS}+\tau_*(\mu_{\rm FS})}=0\,.
\]
\end{lemma}
\begin{proof} Recall that $\mathcal Q_\mu[u_{\mu,*}]=\nu_*(\mu)\,(\tau_*(\mu)+\mu)$ and $\mathcal Q_\mu[u_{(\mu)}]=\nu(\mu)\,(\tau(\mu)+\mu)$, so that
\[
\frac1{\mathcal Q_{\mu_{\rm FS}}[u_{\rm FS}]}\,\frac d{d\mu}\,\mathcal Q_\mu[u_{\mu,*}]_{\left|_{\mu=\mu_{\rm FS}}\right.}=\frac{\nu_*'(\mu_{\rm FS})}{\nu_*(\mu_{\rm FS})}+\frac{1+\tau_*'(\mu_{\rm FS})}{\mu_{\rm FS}+\tau_*(\mu_{\rm FS})}=0
\]
according to~\eqref{TauStar} and~\eqref{NuStar}, and
\[
\frac1{\mathcal Q_{\mu_{\rm FS}}[u_{\rm FS}]}\,\frac d{d\mu}\,\mathcal Q_\mu[u_{(\mu)}]_{\left|_{\mu=\mu_{\rm FS}}\right.}=\frac{\nu'(\mu_{\rm FS})}{\nu(\mu_{\rm FS})}+\frac{1+\tau'(\mu_{\rm FS})}{\mu_{\rm FS}+\tau(\mu_{\rm FS})}\,.
\]
According to~\eqref{Eqn:ExpansionEnergy}, these two quantities are equal, thus proving the result. Alternatively, the identity can be proved directly using the expressions of $\tau'$ and $\nu'$ established in Sections~\ref{Sec:tau} and~\ref{Sec:nu}. This ends the proof of Lemma \ref{Rem2} and of Theorem \ref{Thm:T1bis}, (iii). \end{proof}

\subsection{Reparametrization of the branch for \texorpdfstring{$\theta<1$}{theta<1} and proof of Theorem \ref{Thm:T2}}\label{Sec:theta2} 

Now we are in position to study the local behavior of the branch of the solutions to~\eqref{Euler-thetatheta} parametrized by $\mu$ close to the bifurcation point, that is, for $\mu$ in a neighborhood of $\mu_{\rm FS}$. More precisely, we are interested in the monotonicity of $\mu\mapsto\Lambda^\theta(\mu)$ and the behavior of $\mu\mapsto(\Lambda^\theta(\mu),J^\theta(\mu))$ in a neighborhood of $\mu=\mu_{\rm FS}$. According to the parametrization of Section~\ref{Sec:Intro}, we know that $\Lambda^\theta(\mu)=\theta\,\mu-(1-\theta)\,\tau(\mu)$, so that
\[
(\Lambda^\theta)'=\theta-(1-\theta)\,\tau'
\]
can be computed at $\mu=\mu_{\rm FS}$ using the expression of $\tau'(\mu_{\rm FS})$, that has been computed in Section~\ref{Sec:tau}. Hence we find that
\[
(\Lambda^\theta)'(\mu_{\rm FS})=\theta\,(1+\tau'(\mu_{\rm FS}))-\tau'(\mu_{\rm FS})\,.
\]
\begin{lemma} If $\tau'(\mu_{\rm FS})$ is positive, then we have that $\frac d{d\mu}\,\Lambda^\theta(\mu_{\rm FS})<0$ if and only
\[
\theta<\vartheta_2(p,d):=\frac{\tau'(\mu_{\rm FS})}{1+\tau'(\mu_{\rm FS})}\,.
\]\end{lemma}
Notice that with this definition, $\vartheta_2(p,d)$ is defined for any $p\in(2,2^*)$ and any $d\ge2$. As long as $\vartheta(p,d)<\vartheta_2(p,d)$, $(\Lambda^\theta)'(\mu_{\rm FS})$ is negative if $\theta\in(\vartheta(p,d),\vartheta_2(p,d))$. In all numerical examples that are under consideration in this paper, we find that $\tau'(\mu_{\rm FS})$ is positive. This is of course automatically the case if $c_{p,d}$ itself is positive, because of (ii) in Theorem~\ref{Thm:T1bis}.

We recall that $J^\theta(\mu):=\theta^\theta\,\(\mu+\tau(\mu)\)^\theta\,\nu(\mu)$ and $\Lambda^\theta(\mu)=\theta\,\mu-(1-\theta)\,\tau(\mu)$. Hence
\[
(\log J^\theta)'=\frac{\nu'}{\nu}+\theta\,\frac{1+\tau'}{\mu+\tau}
\]
and a similar formula holds for $J_*^\theta$. At $\mu=\mu_{\rm FS}$, we can use Lemma~\ref{Rem2} and get
\[
(\log J^\theta)'(\mu_{\rm FS})=\frac{\theta\,(1+\tau'(\mu_{\rm FS}))-\tau'(\mu_{\rm FS})}{\mu_{\rm FS}+\tau(\mu_{\rm FS})}=\frac{(\Lambda^\theta)'(\mu_{\rm FS})}{\mu_{\rm FS}+\tau(\mu_{\rm FS})}\,.
\]
Let us define
\[
\delta^\theta:=\frac{(J^\theta)'}{(\Lambda^\theta)'}-\frac{(J_*^\theta)'}{(\Lambda_*^\theta)'}\,.
\]
\begin{lemma}\label{Claim1} Assuming that $c_{p,d}$ is positive, we have
\[
\delta^\theta(\mu_{\rm FS})=0\,.
\]
\end{lemma}
In other words, the parametric curves $\mu\mapsto(\Lambda^\theta(\mu),J^\theta(\mu))$ and $\mu\mapsto(\Lambda_*^\theta(\mu),J_*^\theta(\mu))$ are tangent at $\mu=\mu_{\rm FS}$. It remains to decide the relative position of the two branches in a neighborhood of the bifurcation point. In order to do so, let us define the function
\[
\xi^\theta:=\frac{\big((\Lambda^\theta)'\big)^2}{J^\theta}\,\left[\frac1{(\Lambda^\theta)'}\,\frac d{d\mu}\(\frac{(J^\theta)'}{(\Lambda^\theta)'}\)-\frac1{(\Lambda_*^\theta)'}\,\frac d{d\mu}\(\frac{(J_*^\theta)'}{(\Lambda_*^\theta)'}\)\right]
\]
and discuss the range of positivity of $\xi^\theta$. An elementary computation shows that
\[
\xi^\theta=\frac{(J^\theta)''}{J^\theta}-\frac{(J^\theta)'}{(\Lambda^\theta)'}\,\frac{(\Lambda^\theta)''}{J^\theta}-\frac{(J_*^\theta)''}{J_*^\theta}\(\frac{(\Lambda^\theta)'}{(\Lambda_*^\theta)'}\)^2
\]
because $(\Lambda_*^\theta)''=0$. By definition of $\Lambda^\theta(\mu)=\theta\,\mu-(1-\theta)\,\tau(\mu)$ and $\vartheta_2(p,d)$, we get that $(\Lambda^\theta)''(\mu_{\rm FS})=-\,(1-\theta)\,\tau''(\mu_{\rm FS})$ and $(\Lambda^\theta)'(\mu_{\rm FS})=\frac{\theta-\vartheta_2(p,d)}{1-\vartheta_2(p,d)}$. From~\eqref{TauStar} and $\Lambda_*^\theta(\mu)=\theta\,\mu-(1-\theta)\,\tau_*(\mu)$, we get that $(\Lambda_*^\theta)'=\frac{2\,p\,\theta-(p-2)}{p+2}$. From~\eqref{TauStar} and~\eqref{NuStar}, we know also that $J_*^\theta(\mu)=\kappa_p\,\(2\,p\,\theta/(p+2)\)^\theta\mu^{\theta-\frac{p-2}{2\,p}}$, so that $\frac{(J_*^\theta)'}{J_*^\theta}=\frac{2\,p\,\theta-(p-2)}{2\,p\,\mu}$ and $(\log J_*^\theta)''=-\frac{2\,p\,\theta-(p-2)}{2\,p\,\mu^2}$. As a consequence of Lemma~\ref{Claim1}, we can write that
\[
\frac{(J^\theta)'(\mu_{\rm FS})}{J^\theta(\mu_{\rm FS})}=\frac{(J_*^\theta)'(\mu_{\rm FS})}{J_*^\theta(\mu_{\rm FS})}\,\frac{(\Lambda^\theta)'(\mu_{\rm FS})}{(\Lambda_*^\theta)'(\mu_{\rm FS})}\,.
\]
According to~\eqref{Eqn:ExpansionEnergy}, we have the identity
\[
-\frac 14\,(p^2-4)\,c_{p,d}=\(\log\frac{J^\theta}{J_*^\theta}\)''(\mu_{\rm FS})+(1-\theta)\(\log\frac{\mu+\tau(\mu)}{\mu+\tau_*(\mu)}\)_{|\mu=\mu_{\rm FS}}''\,,
\]
which allows us to compute
\begin{eqnarray*}
&&\kern -18pt\frac{(J^\theta)''(\mu_{\rm FS})}{J^\theta(\mu_{\rm FS})}\\
&=&-\frac 14\,(p^2-4)\,c_{p,d}+\(\frac{(J^\theta)'(\mu_{\rm FS})}{J^\theta(\mu_{\rm FS})}\)^2+(\log J_*^\theta)''(\mu_{\rm FS})\\
&&-(1-\theta)\left[\frac{p+2}{2\,p\,\mu_{\rm FS}}\,\tau''(\mu_{\rm FS})-\(\frac{p+2}{2\,p\,\mu_{\rm FS}}\,\big(1+\tau'(\mu_{\rm FS})\big)\)^2+\frac1{\mu_{\rm FS}^2}\right]\\
&=&-\frac 14\,(p^2-4)\,c_{p,d}+\(\frac{(J_*^\theta)'(\mu_{\rm FS})}{J_*^\theta(\mu_{\rm FS})}\)^2\(\frac{(\Lambda^\theta)'(\mu_{\rm FS})}{(\Lambda_*^\theta)'(\mu_{\rm FS})}\)^2+(\log J_*^\theta)''(\mu_{\rm FS})\\
&&-(1-\theta)\left[\frac{p+2}{2\,p\,\mu_{\rm FS}}\,\tau''(\mu_{\rm FS})-\(\frac{p+2}{2\,p\,\mu_{\rm FS}}\,\big(1+\tau'(\mu_{\rm FS})\big)\)^2+\frac1{\mu_{\rm FS}^2}\right].
\end{eqnarray*}
Because of Lemma~\ref{Claim1}, we also have
\begin{eqnarray*}\nonumber
-\frac{(J^\theta)'(\mu_{\rm FS})}{(\Lambda^\theta)'(\mu_{\rm FS})}\,\frac{(\Lambda^\theta)''(\mu_{\rm FS})}{J^\theta(\mu_{\rm FS})}&=&-\frac{(J_*^\theta)'(\mu_{\rm FS})}{J_*^\theta(\mu_{\rm FS})}\,\frac{(\Lambda^\theta)''(\mu_{\rm FS})}{(\Lambda_*^\theta)'(\mu_{\rm FS})}\\
&=&\frac{p+2}{2\,p\,\mu_{\rm FS}}\,(1-\theta)\,\tau''(\mu_{\rm FS})\,.
\label{xi3}
\end{eqnarray*}
Collecting the above identities, we can compute the value of $\xi^\theta(\mu_{\rm FS})$ as
\begin{eqnarray*}
\xi^\theta(\mu_{\rm FS})&=&-\frac 14\,(p^2-4)\,c_{p,d}\\
&&+\,(\log J_*^\theta)''(\mu_{\rm FS})\left[1-\(\frac{(\Lambda^\theta)'(\mu_{\rm FS})}{(\Lambda_*^\theta)'(\mu_{\rm FS})}\)^2\right]\\
&&-\frac{1-\theta}{\mu_{\rm FS}^2}\left[1-\(\frac{p+2}{2\,p}\,\big(1+\tau'(\mu_{\rm FS})\big)\)^2\right].
\end{eqnarray*}
The cancellation of the terms involving $\tau''(\mu_{\rm FS})$ is a remarkable fact. By definition of $\vartheta_2(p,d)$, we get
\begin{eqnarray*}
\xi^\theta(\mu_{\rm FS})&=&-\frac 14\,(p^2-4)\,c_{p,d}\\
&&-\,\frac{2\,p\,\theta-(p-2)}{2\,p\,\mu_{\rm FS}^2}\left[1-\(\frac{p+2}{2\,p\,\theta-(p-2)}\,\frac{\theta-\vartheta_2(p,d)}{1-\vartheta_2(p,d)}\)^2\right]\\
&&-\frac{1-\theta}{\mu_{\rm FS}^2}\left[1-\(\frac{p+2}{2\,p}\,\frac 1{1-\vartheta_2(p,d)}\)^2\right]
\end{eqnarray*}
and finally arrive at
\[
\xi^\theta(\mu_{\rm FS})=-\frac 14\,(p^2-4)\,c_{p,d}+\frac{p+2}{4\,p^2\,\mu_{\rm FS}^2}\,\frac{(1-\theta)\,\big(2\,p\,\vartheta_2(p,d)-(p-2)\big)^2}{\big(1-\vartheta_2(p,d)\big)^2\,\big(2\,p\,\theta-(p-2)\big)}\,.
\]
At this point, we can observe that $\vartheta(p,d)\ge\frac{p-2}{2\,p}$. The reader is then invited to check that the function $\theta\mapsto\xi^\theta(\mu_{\rm FS})$ is nonincreasing on $[\vartheta(p,d),1]$ and
\[
\xi^{\vartheta_2(p,d)}(\mu_{\rm FS})=-\frac 14\,(p^2-4)\,c_{p,d}+\frac{p+2}{4\,p^2\,\mu_{\rm FS}^2}\,\frac{2\,p\,\vartheta_2(p,d)-(p-2)}{1-\vartheta_2(p,d)}=0
\]
because of (ii) in Theorem~\ref{Thm:T1bis}. Recall that the positivity of $c_{p,d}$ is required in \eqref{positivityofcpd}. 

We have then proved that if $\vartheta_2(p,d)\le\vartheta(p,d)$, $\xi^\theta(\mu_{\rm FS})$ is negative for any $\theta\in(\vartheta(p,d),1]$. Otherwise, if $c_{p,d}$ is positive, $\xi^\theta(\mu_{\rm FS})$ is positive if $\theta\in[\vartheta(p,d),\vartheta_2(p,d))$ and negative if $\theta\in(\vartheta_2(p,d),1]$.

The expansion \eqref{Eqn:Expansion} and the subsequent computations are valid, and make sense for the approximation of a local minimizer of $\mathcal Q_\Lambda^\theta$, as soon as the coefficient $c_{p,d}$, whose expression is established in Section~\ref{Sec:Qmu}, is positive. Then for any $\theta\in(\max\{\vartheta(p,d),\vartheta_2(p,d)\},1)$, the curve of the energies of the non-symmetric solutions of the Euler-Lagrange equations is concave, nondecreasing as a function of $\Lambda$ in a neighborhood of the bifurcation point, and below the energies of the symmetric functions. If $\vartheta_2(p,d)>\vartheta(p,d)$, then the curve of the energies of the non-symmetric solutions is above the energies of symmetric functions in a neighborhood of the bifurcation point if $\theta\in[\vartheta(p,d),\vartheta_2(p,d))$. Practically,  whether $c_{p,d}$ is positive or not relies either on the sufficient condition given in Theorem \ref{Thm:Check} or on numerical computations. However, the estimate of Section~\ref{Sec:Check} shows that this occurs at least in a large subinterval of $(2,2^*)$. This completes the proof of Theorem~\ref{Thm:T2}.\\ \finprf

\section{Numerical results and the two scenarios}\label{Sec:NumericsAndScenarios}

\subsection{Symmetric and non-symmetric branches, and their asymptotic behavior}\label{Sec:NumKnown}

In \cite{Freefem} the branches of solutions which bifurcate from the branches of symmetric solutions at the smallest possible value of $\Lambda$ have been computed numerically. For completeness, we start by presenting some of these numerical results, which are the main motivation of the present paper. The branch of symmetric solutions is explicit. The branch $\mu\mapsto(\Lambda^\theta(\mu),J^\theta(\mu))$ bifurcates from the symmetric ones at $\mu=\mu_{\rm FS}$ and is computed numerically. The algorithm is based on descent techniques and on an iteration scheme which allows us to compute the branches of solutions by continuation. We carried out the computations for dimension $d=5$ and various values of $p$ and~$\theta$. We have of course no guarantee that the solutions that we have computed are the optimal ones, but at least the values that we have found are fully compatible with what is expected for theoretical reasons. In particular, the curve of the computed estimates of the best constant is an increasing function of $\Lambda$ with the right convexity properties, which can reasonably be expected to coincide with $\Lambda\mapsto\mathsf K_{\rm CKN}(\theta,\Lambda,p)$. Moreover, when $\theta$ approaches $\vartheta(p,d)$ from above, the curve $\Lambda\mapsto\mathsf K_{\rm CKN}(\theta,\Lambda,p)$ approaches $\Lambda\mapsto\max\{\mathsf K_{\rm CKN}^*(\theta,\Lambda,p),\mathsf K_{\rm GN}\}$ when $\mathsf K_{\rm GN}>\mathsf K_{\rm CKN}^*(\vartheta(p,d),\Lambda_{\rm FS}(p,\vartheta(p,d)),p)$. Last but not least, the asymptotics predicted in Theorem~\ref{Thm:asymptott} are not only correct (dotted lines in Figs~\ref{fig1}--\ref{fig3}) but provide a good upper estimate of the curve in the whole range $\Lambda>0$. 
\begin{figure}[ht]\begin{center}
\includegraphics[width=8cm]{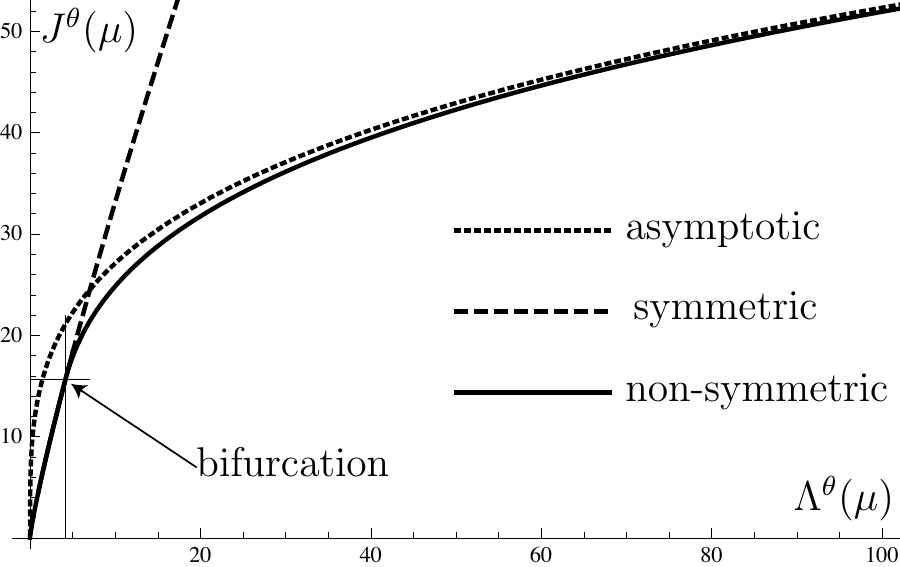}
\caption{\label{fig1} Parametric plot of $\mu\mapsto(\Lambda^\theta(\mu),J^\theta(\mu))$ for $p=2.8$, $d=5$, $\theta=1$. Non-symmetric solutions bifurcate from symmetric ones at a bifurcation point $\mu=\mu_{\rm FS}$ computed by V.~Felli and M.~Schneider. The branch behaves for large values of $\Lambda$ as predicted by F.~Catrina and Z.-Q.~Wang.}
\end{center}\end{figure}
\begin{figure}[ht]\begin{center}
\includegraphics[width=8cm]{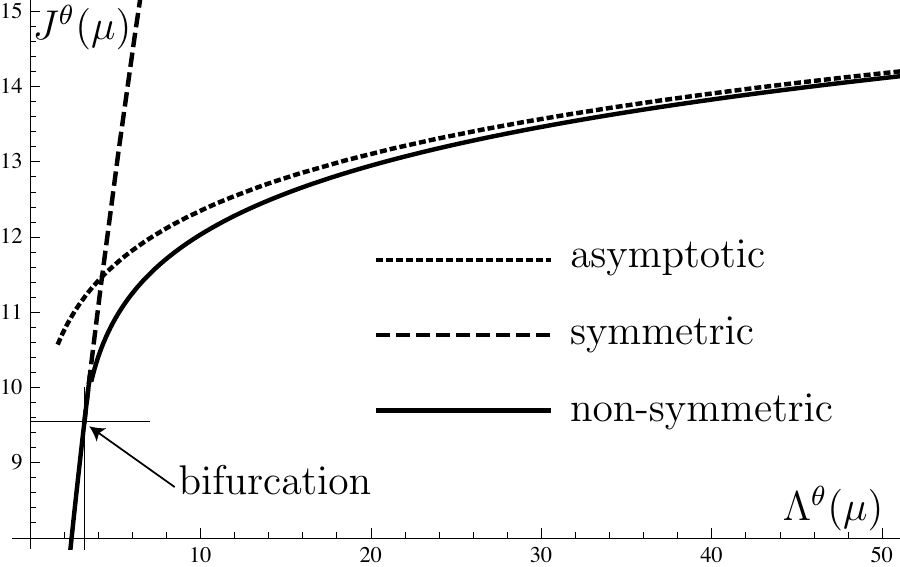}
\caption{\label{fig2} Parametric plot of $\mu\mapsto(\Lambda^\theta(\mu),J^\theta(\mu))$ for $p=2.8$, $d=5$, $\theta=0.8$. The behavior is similar to the case $\theta=1$ up to the reparametrization described in Section~\ref{Sec:Intro}, while the asymptotic behavior of the branch for large values of $\Lambda$ is in agreement with the results of Section~\ref{Sec:Gagliardo-Nirenberg}.}
\end{center}\end{figure}
\begin{figure}[ht]\begin{center}
\includegraphics[width=8cm]{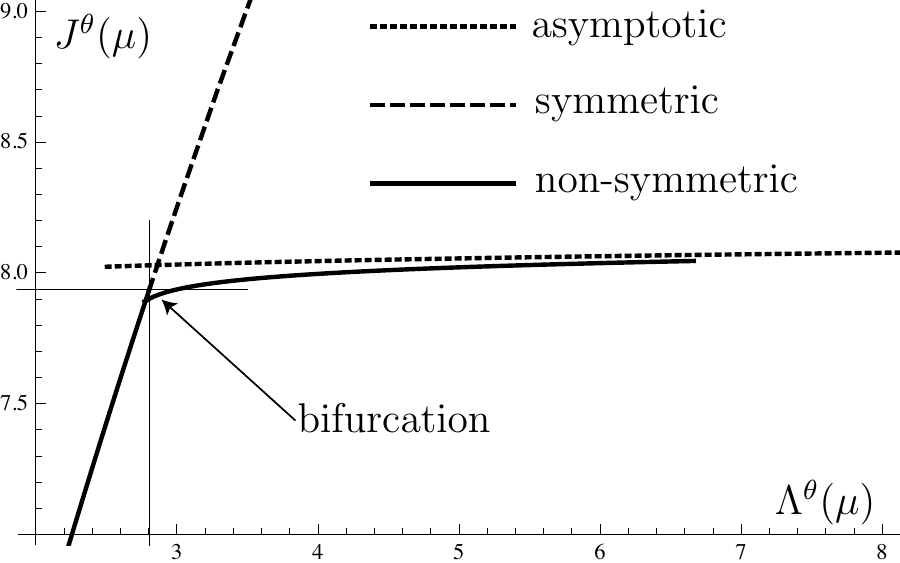}
\caption{\label{fig3} Parametric plot of $\mu\mapsto(\Lambda^\theta(\mu),J^\theta(\mu))$ for $p=2.8$, $d=5$, $\theta=0.72>\vartheta(2.8,5)\approx0.7143$. The asymptotic behavior of the branch for large values of $\Lambda$ is in agreement with the results of Section~\ref{Sec:Gagliardo-Nirenberg}. We may however notice that the bifurcation seems not to occur where expected. A more detailed computation in the Section~\ref{Sec:BifurcationsQualitativeDependence} sheds light on this phenomenon.}
\end{center}\end{figure}

\subsection{Two scenarios}\label{Sec:TwoScenarii}

The branch of symmetric minimal solutions, $(\Lambda, u_\Lambda^*)$, is known explicitly and is monotone in energy, that is, the function $\Lambda\mapsto \mathcal Q_\Lambda^\theta(u_\Lambda^*)$ is monotone increasing in $(0,+\infty)$. In the computations described in \cite{Freefem} we observe that the branch of non-symmetric solutions, $(\Lambda^\theta(\mu), \mathcal Q_\mu^\theta[u_\mu])$, is monotone for some values of~$\theta$ (for instance for $\theta=1$), but not always. More concretely, for certain values of $p$ and~$d$, the numerical results show that there exists an exponent $\vartheta_1=\vartheta_1(p,d)\in(\vartheta(p,d),1)$ such that for any $\theta\in[\vartheta_1, 1]$ the branch $(\Lambda^\theta(\mu),\mathcal Q_\mu^\theta[u_\mu])$ is monotone increasing. But when $\vartheta(p,d)<\theta<\vartheta_1(p,d)$, a dramatic change occurs: see Figs.~\ref{fig5} and~\ref{fig6}. For the values of $p$ and $d$ that have been considered numerically in those figures, the branch is not monotone anymore for $\mu>\mu_{\rm FS}$, thus producing non-symmetric solutions and candidates for optimal functions in the inequalities for values of $\Lambda<\Lambda^\theta(\mu_{\rm FS})$. This phenomenon provides an explanation for the results proved in \cite{springerlink:10.1007/s00526-011-0394-y} using rigorous \emph{a priori} estimates.

The limiting case $\theta=\vartheta(p,d)$ is very interesting: see Figs.~\ref{fig6} and~\ref{fig7r}. Let us define $\Lambda_{\rm GN}^*(p,d)$ and $\mu_{\rm GN}$ such that
\[
\mathsf K_{\rm GN}=\mathsf K_{\rm CKN}^*(\vartheta(p,d),\Lambda_{\rm GN}^*(p,d),p)\quad\mbox{and}\quad\Lambda_*^{\vartheta(p,d)}(\mu_{\rm GN})=\Lambda_{\rm GN}^*(p,d)\,.
\]
Whether $\mathsf K_{\rm GN}$ is larger or smaller than $\mathsf K_{\rm CKN}^*(\vartheta(p,d),\Lambda_{\rm FS}(p,\vartheta(p,d)),p)$ determines, at least in the framework of our computations, whether $\theta$ is smaller than $\vartheta_1(p,d)$ or not. This has been observed in \cite{Oslo} and theoretical consequences have been established in~\cite{springerlink:10.1007/s00526-011-0394-y}, in the limit regime $p\to2$. Before going further, let us observe that $\vartheta_1(p,d)$ is an exponent associated with a global property of the branch. 

Based on our numerical computations, we are now in position to formulate the following alternative.

\noindent\emph{Scenario~1.} If $\mathsf K_{\rm GN}\le\mathsf K_{\rm CKN}^*(\vartheta(p,d),\Lambda_{\rm FS}(p,\vartheta(p,d)),p)$, then for all $\theta\ge\vartheta(p,d)$, the optimal functions are symmetric for any $\Lambda\in(0,\Lambda_{\rm FS}(p,\theta)]$ and the branch of non-symmetric solutions is optimal for any $\Lambda>\Lambda_{\rm FS}(p,\theta)$. Such solutions exist for arbitrarily large values of $\Lambda$ if $\theta>\vartheta(p,d)$: see Figs.~\ref{fig7l} and~\ref{fig7c}, but may exist only for a finite range of $\Lambda$ if $\theta=\vartheta(p,d)$ and $\mathsf K_{\rm GN}<\mathsf K_{\rm CKN}^*(\vartheta(p,d),\Lambda_{\rm FS}(p,\vartheta(p,d)),p)$: see Fig.~\ref{fig7r}.

\noindent\emph{Scenario~2.} If $\mathsf K_{\rm GN}>\mathsf K_{\rm CKN}^*(\vartheta(p,d),\Lambda_{\rm FS}(p,\vartheta(p,d)),p)$, there exists $\vartheta_1=\vartheta_1(p,d)\in(\vartheta(p,d),1)$ such that for any $\theta\in[\vartheta_1, 1]$ the branch is monotone increasing. We further observe numerically that $\vartheta_1=\vartheta_2$ (see Fig.~\ref{fig10}), where $\vartheta_2$ has been defined in Section~\ref{Sec:theta2}: for any $\theta\in[\vartheta(p,d),\vartheta_2(p,d))$, we know that the branch of non-symmetric functions is decreasing in a neighborhood of the bifurcation point, but also has a larger energy than the symmetric solutions of the Euler-Lagrange equation (for the same value of~$\Lambda$). Hence, in the critical case $\theta=\vartheta(p,d)$ we have $\Lambda_{\rm GN}^*(p,d)<\Lambda_{\rm FS}(p,\vartheta(p,d))$ and for any $\Lambda$ in the range $(0,\Lambda_{\rm GN}^*]:=\{\Lambda_*^{\vartheta(p,d)}(\mu)\,:\,\mu\in(0,\mu_{\rm GN}]\}$, the optimal functions are symmetric and $\mathsf K_{\rm CKN}(\vartheta(p,d),\Lambda,p)=\mathsf K_{\rm CKN}^*(\vartheta(p,d),\Lambda,p)$. See Fig.~\ref{fig6}. Moreover, in the language of the concentration-compactness method, according to \cite{1005}, for any $\Lambda>\Lambda_{\rm GN}^*(p,d)$ the optimal constant is determined by the \emph{problem at infinity} and $\mathsf K_{\rm CKN}(\vartheta(p,d),\Lambda,p)=\mathsf K_{\rm GN}$. From the viewpoint of the reparametrized branch, we numerically observe that both $\mu\mapsto\Lambda^{\vartheta(p,d)}(\mu)$ and $\mu\mapsto J^{\vartheta(p,d)}(\mu)$ are decreasing for $\mu>\mu_{\rm FS}$, at least for the values of $p$ for which computations have been done. In the subcritical case corresponding to $\vartheta(p,d)<\theta<\vartheta_2(p,d)$, the reparametrized branch $(\Lambda^\theta,J^\theta)$ is not monotone. Numerically we observe that it is monotone for $\theta>\vartheta_2(p,d)$, hence supporting the observed fact that $\vartheta_1(p,d)=\vartheta_2(p,d)$ (see Fig.~\ref{fig10}).

Altogether, based on our numerical observations, what decides between \emph{Scenario~1} and \emph{Scenario~2} is the relative value of $\mathsf K_{\rm GN}$ and $\mathsf K_{\rm CKN}^*(\vartheta(p,d),\Lambda_{\rm FS}(p,\vartheta(p,d)),p)$. Equality of these two optimal constants determines a value of $p=p^\star(d)$. Numerically we find that $p^\star(5)\approx3.001$ and $\vartheta(p^\star(5),5)=\vartheta_1(p^\star(5),5)\approx0.834$. For $p\in[p^\star(d),2^*)$, only \emph{Scenario~1} occurs (numerical observation). For $p\in(2,p^\star(d))$ we have $\vartheta(p,d)>\vartheta_1(p,d)$ and \emph{Scenario~2} occurs. More precisely, the fact that the branch cannot be globally monotone increasing if $\theta<\vartheta_2(p,d)$ is a consequence of Section~\ref{Sec:theta2} while the fact that the branch is monotone increasing if $\theta>\vartheta_2(p,d)$ is a numerical observation. These results are fully consistent with the ones of \cite{springerlink:10.1007/s00526-011-0394-y} and \cite{Oslo}. Now let us give some details.

\subsection{Bifurcations and qualitative dependence in \texorpdfstring{$\theta$}{theta}}\label{Sec:BifurcationsQualitativeDependence}

In Fig.~\ref{fig3}, a careful inspection shows that the symmetric and the non-symmetric branches of solutions differ for values of $\Lambda$ strictly less than $\Lambda^\theta(\mu_{\rm FS})$. This is not the case for $\theta$ close enough to $1$: see Fig.~\ref{fig4}, but very clear on Fig.~\ref{fig5}.
\begin{figure}[hb]\begin{center}
\includegraphics[height=5.5cm, width=6.8cm]{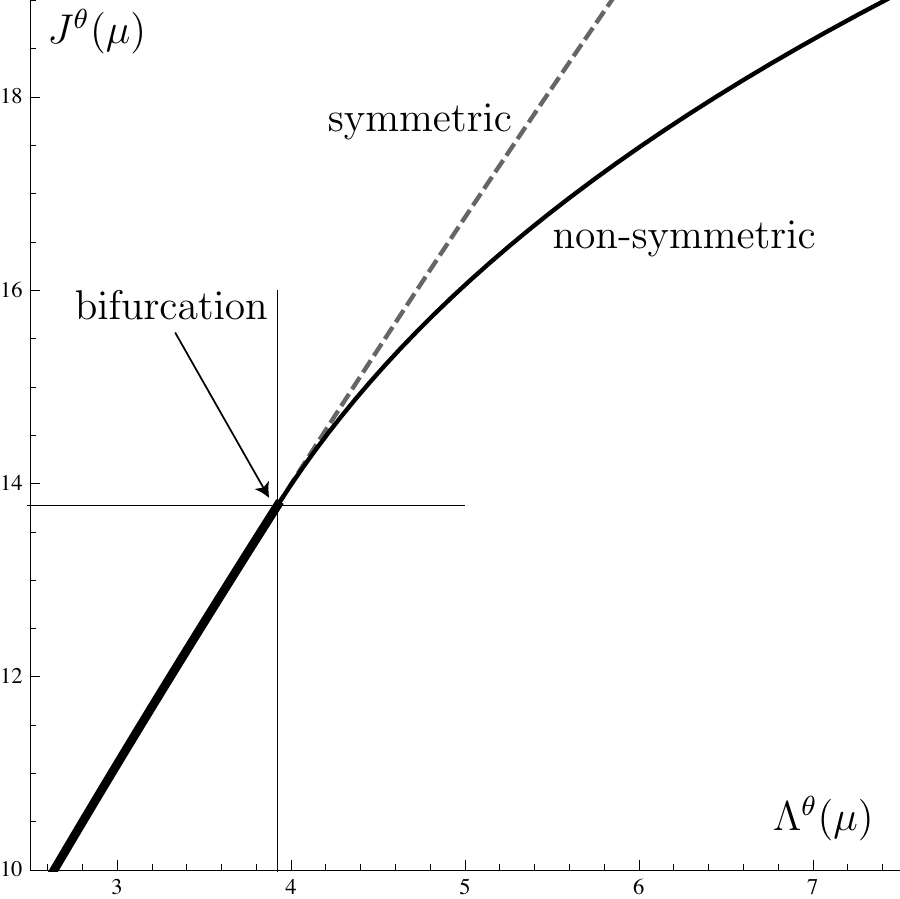}
\caption{\label{fig4} Parametric plot of $\mu\mapsto(\Lambda^\theta(\mu),J^\theta(\mu))$ for $p=2.8$, $d=5$, $\theta=0.95$ close to the bifurcation point. For $\Lambda\le\Lambda^\theta(\mu_{\rm FS})$, all solutions are symmetric, while for $\Lambda>\Lambda^\theta(\mu_{\rm FS})$, non-symmetric solutions provide a better constant in the interpolation inequalities, exactly as for the case $\theta=1$.}
\end{center}\vspace*{-12pt}\end{figure}
\begin{figure}[ht]\begin{center}
\includegraphics[height=6cm,width=7.5cm]{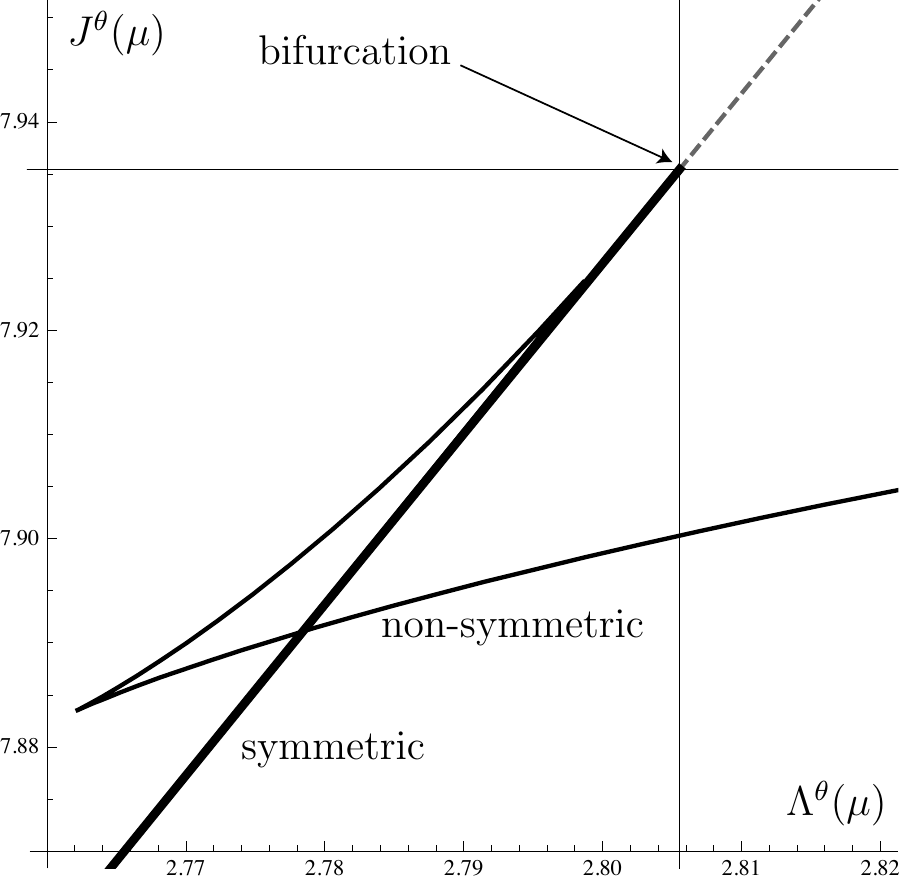}
\caption{\label{fig5} Parametric plot of $\mu\mapsto(\Lambda^\theta(\mu),J^\theta(\mu))$ for $p=2.8$, $d=5$, $\theta=0.72$ close to the bifurcation point. Non-symmetric solutions exist for $\Lambda<\Lambda^\theta(\mu_{\rm FS})$. There exists a value $\Lambda_{\rm GN}^*<\Lambda^\theta(\mu_{\rm FS})$ such that optimal functions are symmetric for any $\Lambda\in(0,\Lambda_{\rm GN}^*)$ and are non-symmetric for $\Lambda>\Lambda_{\rm GN}^*$. When $\Lambda=\Lambda_{\rm GN}^*$, symmetric and non-symmetric optimal functions co-exist.}
\end{center}\vspace*{-12pt}\end{figure}
When $\theta$ approaches $\vartheta(p,d)$, the branch (locally) converges to its limit: see Fig.~\ref{fig6}.
\begin{figure}[ht]\begin{center}
\includegraphics[height=6cm,width=7.5cm]{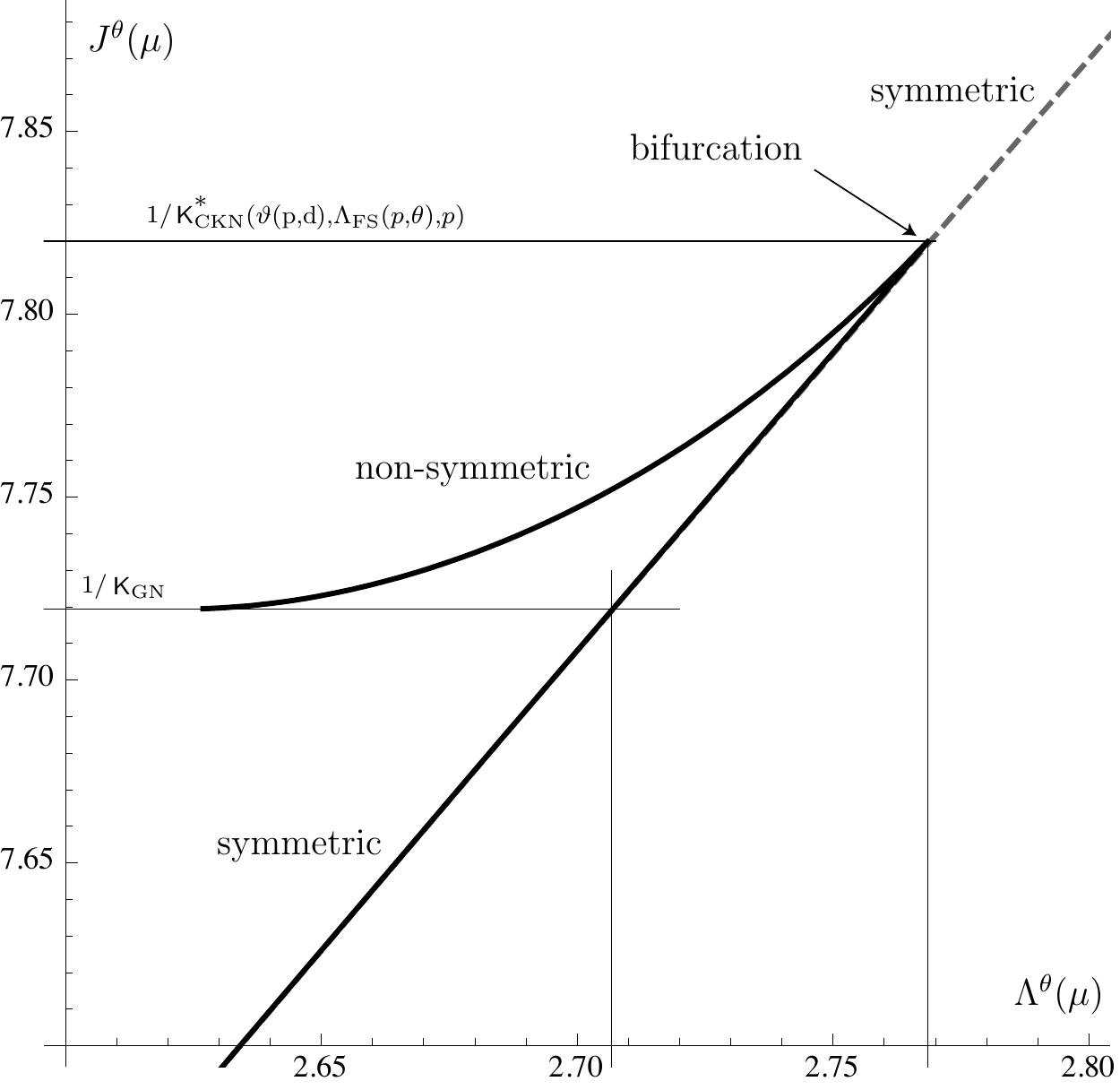}
\caption{\label{fig6} Critical case $\theta=\vartheta(p,d)$. Parametric plot of $\mu\mapsto(\Lambda^\theta(\mu),J^\theta(\mu))$ for $p=2.8$, $d=5$, $\theta=\vartheta(2.8,5)\approx0.7143$ close to the bifurcation point. Non-symmetric solutions exist for $\Lambda<\Lambda^\theta(\mu_{\rm FS})$ but, at least for the chosen values of $p$ and $d$, are never optimal functions for the interpolation inequalities. There exists a value $\Lambda_{\rm GN}^*<\Lambda^\theta(\mu_{\rm FS})$ such that optimal functions exists and are symmetric for any $\Lambda\in(0,\Lambda_{\rm GN}^*]$ and do not exist for $\Lambda>\Lambda_{\rm GN}^*$. Moreover, $\Lambda_{\rm GN}^*=\Lambda^\theta(\mu_{\rm GN})$ with the notations of Section~\ref{Sec:Gagliardo-Nirenberg} and $\mathsf K_{\rm CKN}(\vartheta(p,d),\Lambda,p)=\mathsf K_{\rm GN}$ for any $\Lambda>\Lambda_{\rm GN}^*$.}
\end{center}\end{figure}
The figures~\ref{fig5} and~\ref{fig6} correspond to \emph{Scenario~2} (for $\vartheta\le\theta<\vartheta_1$), that is, to the case
\[
\mathsf K_{\rm GN}>\mathsf K_{\rm CKN}^*(\vartheta(p,d),\Lambda_{\rm FS}(p,\vartheta(p,d)),p)\,.
\]
In other words, this means that $p<p^\star(d)$. 

The case $p>p^\star(d)$, \emph{i.e.,~Scenario~1,} also occurs and corresponding plots are shown in Figs.~\ref{fig7l}--\ref{fig7r}. There we take $d=5$, $p=3.15\ge p^\star(5)\approx3.001$.
\begin{figure}[ht]\begin{center}
\includegraphics[width=7.5cm]{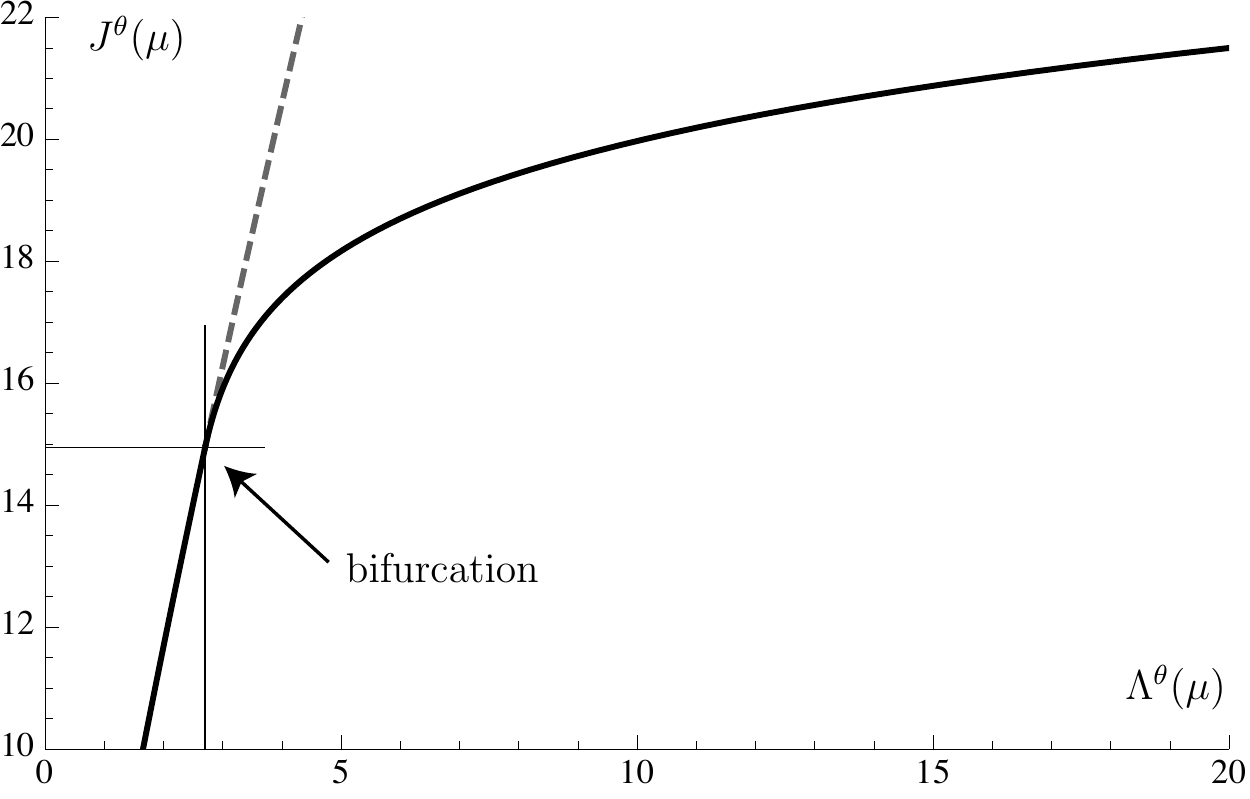}
\caption{\label{fig7l} Parametric plot of $\mu\mapsto(\Lambda^\theta(\mu),J^\theta(\mu))$ for $p=3.15$, $d=5$, $\theta=1$. The behavior of the non-symmetric branch is similar to the one found in Fig.~\ref{fig1}.}
\end{center}\end{figure}
\begin{figure}[ht]\begin{center}
\includegraphics[width=7.5cm]{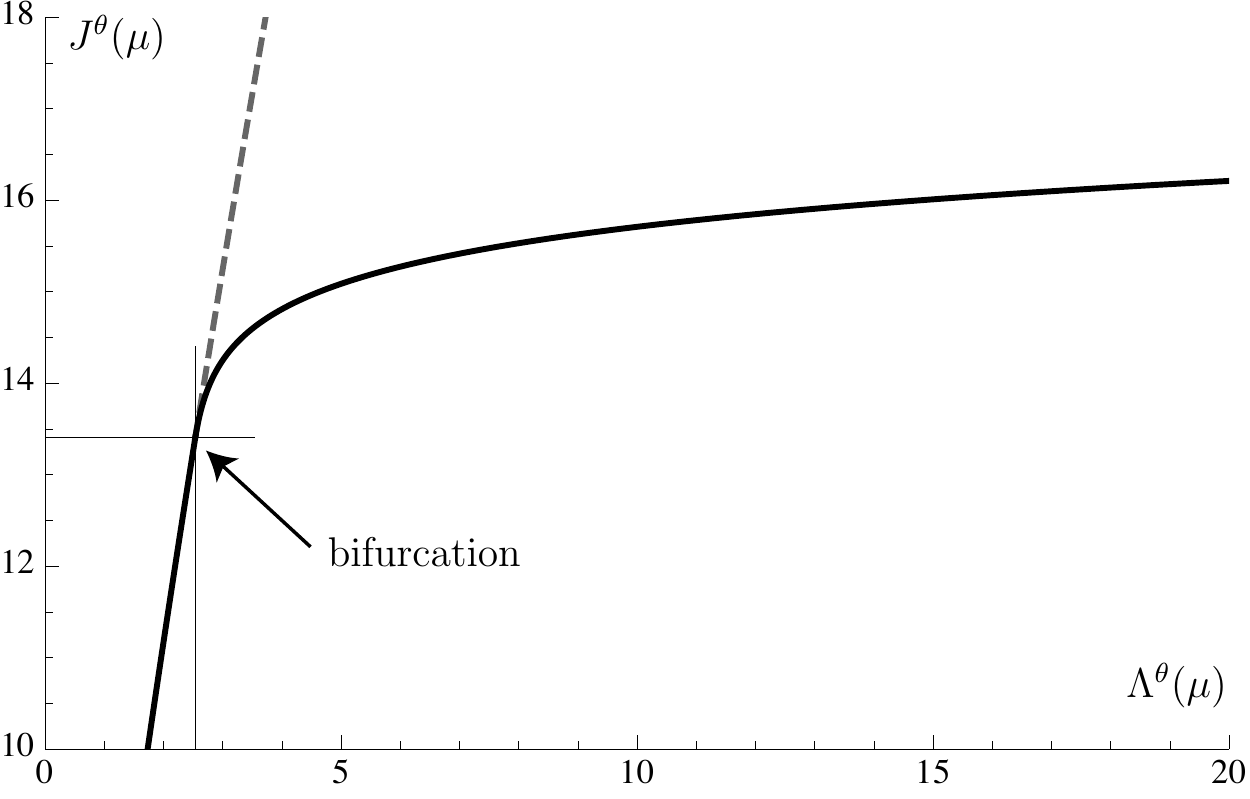}
\caption{\label{fig7c} Parametric plot of $\mu\mapsto(\Lambda^\theta(\mu),J^\theta(\mu))$ for $p=3.15$, $d=5$, $\theta=0.95$. The behavior of the non-symmetric branch is still similar to the one found in Fig.~\ref{fig4}.}
\end{center}\end{figure}
\begin{figure}[ht]\begin{center}
\includegraphics[width=7.5cm]{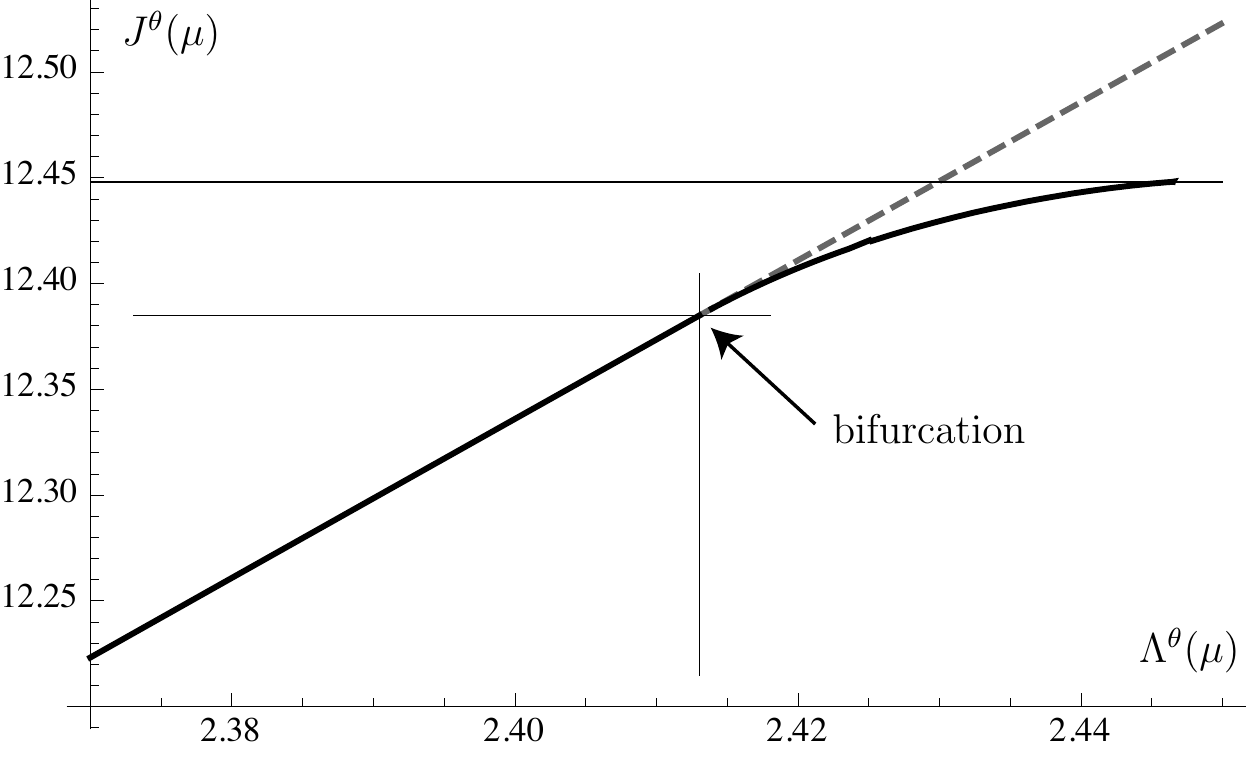}
\caption{\label{fig7r} Parametric plot of $\mu\mapsto(\Lambda^\theta(\mu),J^\theta(\mu))$ for $p=3.15$, $d=5$, $\theta=\vartheta(3.15,5)\approx0.9127$, that is for the critical value of $\theta$.}
\end{center}\end{figure}
\clearpage In Section~\ref{Sec:Bifurcation} we proved that the symmetric and the non-symmetric branches of solutions are always tangent at $\Lambda_{\rm FS}$. What happens in a neighborhood of the bifurcation point is therefore difficult to decide in view of the plots of the branches, especially when $\theta$ is in a neighborhood of $\vartheta_1(p,d)$. To illustrate this difficulty, we may for instance observe that figure Fig.~\ref{fig5} is an enlargement of Fig.~\ref{fig3}. Hence we have to discard the possibility of other scenarios than the ones described in Section~\ref{Sec:TwoScenarii} at least in a neighborhood of the bifurcation point.

The computations of Section \ref{Sec:Bifurcation} have been done for the approximation $u_{(\mu)}$ of the solution $u_\mu$. We expect that the estimates converge as $\mu\to(\mu_{\rm FS})_+$ and this is what is observed numerically. Whether $c_{p,d}$ is positive has been discussed in Section~\ref{Sec:Check}, but can be checked numerically: we know that $c_{p,d}$ is positive and finite as long as $c_{p,d}^{\rm approx}$ is positive (see Fig.~\ref{fig9}, and Fig.~\ref{fig11} for a discussion of the sign of $c_{p,d}^{\rm approx}$), and numerically we find that $c_{p,d}$ is always positive.
\begin{figure}[ht]\begin{center}
\includegraphics[width=7cm]{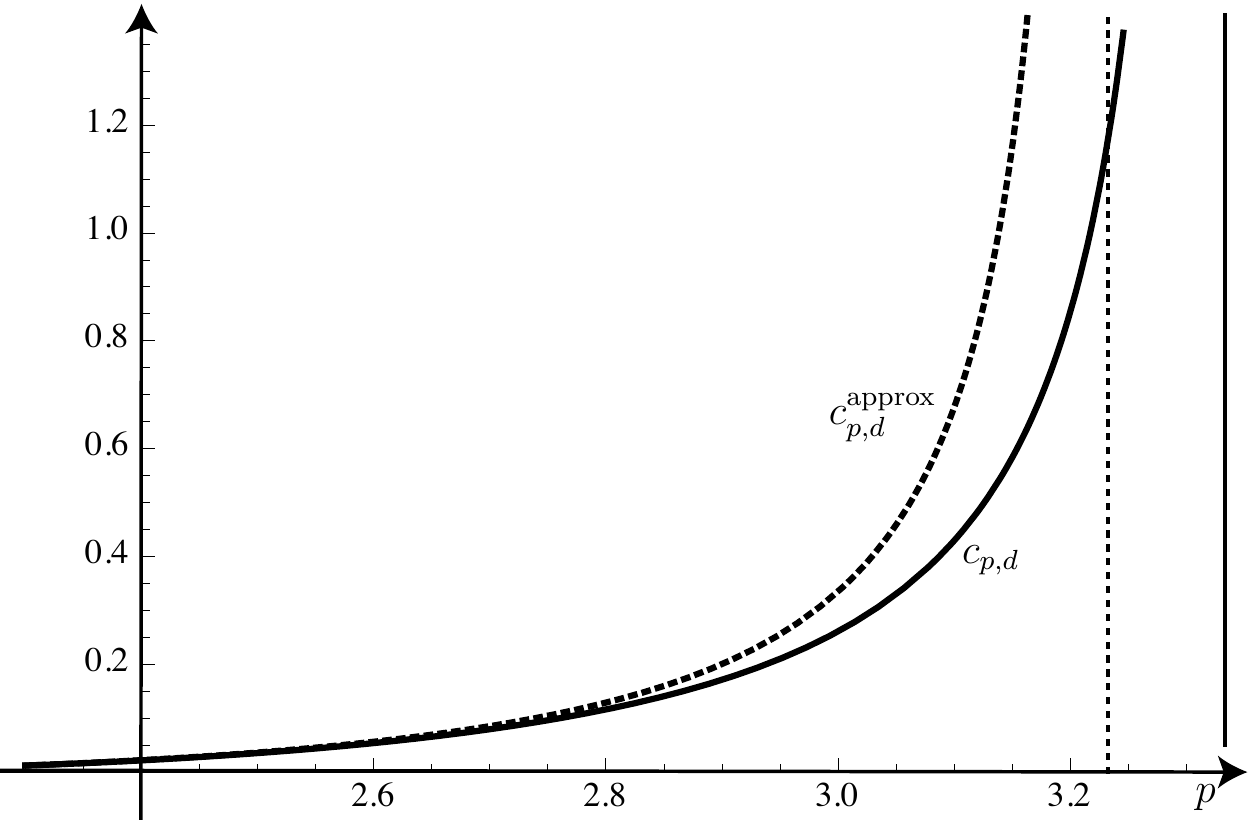}
\caption{\label{fig9} Computation of $c_{p,d}$ with $d=5$ as a function of $p$. We observe that the numerical solution is positive for any $p\in(2,10/3)$ where $10/3$ is the critical exponent (corresponding to the plain vertical line). The estimate of Section~\ref{Sec:Check} corresponds to the dotted line and holds for $p\lesssim3.2323$ (corresponding to the dotted vertical line).}
\end{center}\vspace*{-12pt}\end{figure}
\begin{figure}[ht]\begin{center}
\includegraphics[width=7cm]{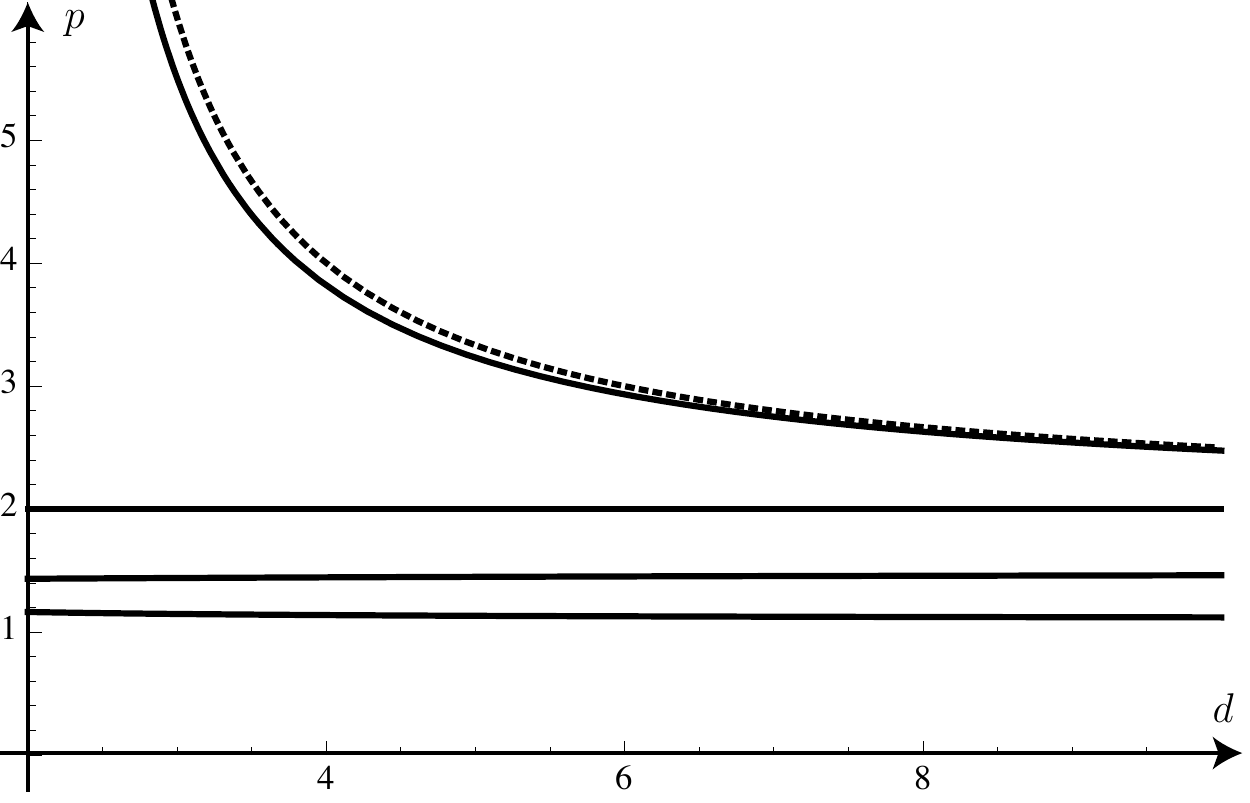}
\caption{\label{fig11} The plain curves are the zeros of $c_{p,d}^{\rm approx}$ as $d$ varies in the interval $(2,10)$, except $p=2$ which is a singularity. The dotted curve is given by $d\mapsto2\,d/(d-2)=2^*$. For a given $d$, $c_{p,d}^{\rm approx}$ is therefore positive for a large subinterval in $p$ of $(2,2^*)$.}
\end{center}\end{figure}

Under the above precautions, we know from Section \ref{Sec:Bifurcation} that there exists a number $\vartheta_2(p,d)$ such that the behavior of the branch in a neighborhood of the bifurcation point $\mu=\mu_{\rm FS}$ discriminates between two regimes corresponding to $\theta>\vartheta_2(p,d)$ and $\theta<\vartheta_2(p,d)$. When $\theta<\vartheta_2(p,d)$, we have $(\Lambda^\theta)'(\mu_{\rm FS})<0$ and the contrary happens when $\theta>\vartheta_2(p,d)$. So, locally, the reparametrized branch is on the right of the bifurcation point and $J^\theta$ is a monotone increasing function of $\Lambda$ (at least when $\mu$ is in a right neighborhood of $\mu_{\rm FS}$) if and only if $\theta>\vartheta_2(p,d)$. Since global monotonicity implies local monotonicity near the bifurcation point, if the numerical computations of the branches are consistent with the study of the bifurcation carried out in Section \ref{Sec:Bifurcation}, then we should have that $\vartheta_1(p,d)=\vartheta_2(p,d)$. It is not easy to establish a qualitative property such as the monotonicity, but at least we observe in Fig.~\ref{fig10} that for $\theta=\vartheta(p,d)$ the range in $p$ for which $\vartheta_2(p,d)\ge\vartheta(p,d)$ corresponds to the range in $p$ for which the Gagliardo-Nirenberg constant compares well with the energy at the bifurcation point.
\begin{figure}[ht]\begin{center}
\includegraphics[width=9cm]{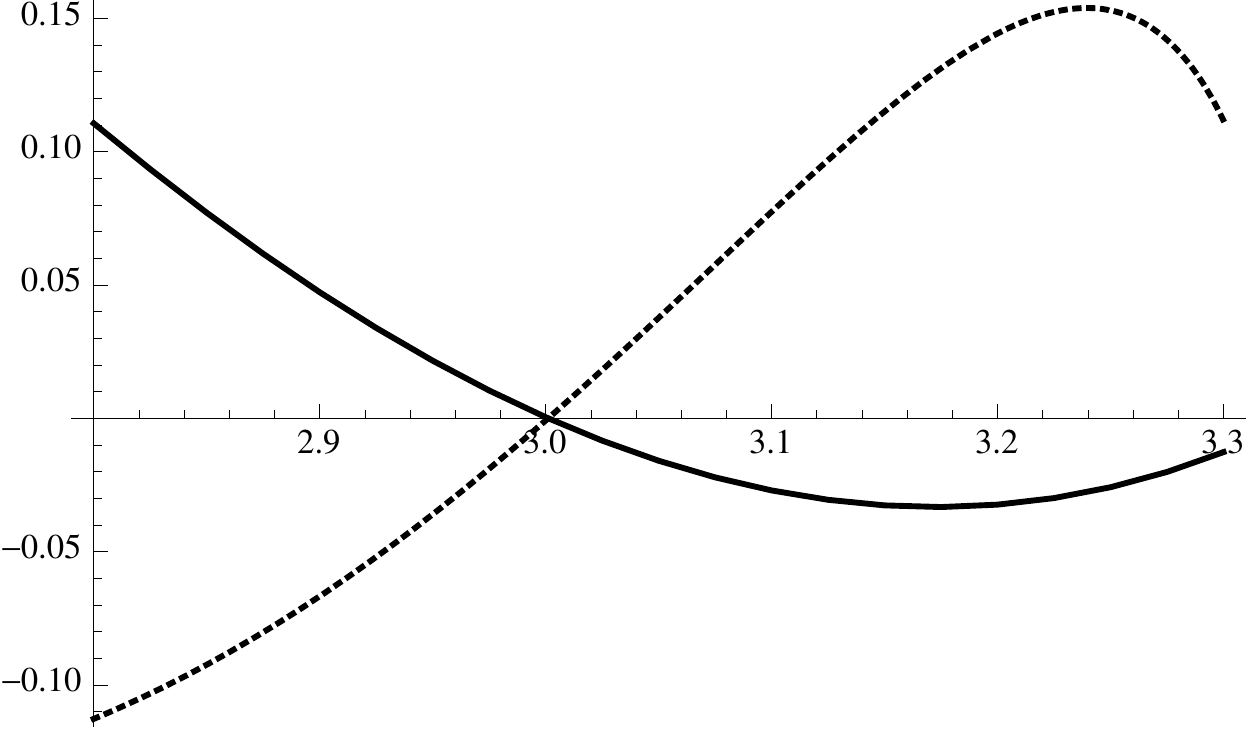}
\caption{\label{fig10} Comparison of the local and asymptotic criteria in the critical case $\theta=\vartheta(p,d)$ when $d=5$. The dotted curve corresponds to the function $p\mapsto1/\mathsf K_{\rm GN}(p,d)-J^{\vartheta(p,d)}(\mu_{\rm FS})$, that is, the difference of the asymptotic energy of the branch and the energy at the bifurcation point: when it is negative, this means that $\vartheta_1(p,d)$ is defined and larger than $\vartheta(p,d)$, so that \emph{Scenario~2} takes place. When it is positive, this means that \emph{Scenario~1} can be expected. The exponent $\vartheta_2(p,d)$ can be defined for any $p\in(2,2^*)$. The plain curve represents $p\mapsto5\,(\vartheta_2(p,d)-\vartheta(p,d))$ and positivity indicates that, at least locally around the bifurcation point, \emph{Scenario~2} takes place. Hence the local (around the bifurcation point) and asymptotic (as $\Lambda\to+\infty$) criteria coincide.}
\end{center}\end{figure}

\section{Concluding remarks}\label{Sec:Conclusion}

In this paper we have established the asymptotic behavior of the branches for all $\theta\in (\vartheta(p,d), 1]$. There is a good agreement between this behavior and that of the numerical branches, which reinforces the conjecture that the computed branches contain the extremals for the Caffarelli-Kohn-Nirenberg inequalities.

We have also studied the precise behavior of the branches of non-symmetric solutions near the first possible bifurcation point on the symmetric branch within the framework of a particular \emph{Ansatz} defined by \eqref{Eqn:Expansion}. By doing so, we have obtained the existence of a critical exponent $\vartheta_2(p,d)$ above which the branch is monotone, increasing and potentially optimal, and below which it is certainly not optimal in a neighborhood of the bifurcation point.

The final result of this paper is the comparison of the above criterion based on the local behavior of the branch near the bifurcation point and the criterion based on the asymptotic energy of the branch in the critical case $\theta=\vartheta(p,d)$, using numerical methods. They coincide, which gives solid grounds to the alternative that has been numerically observed:\\
\emph{Scenario~1.} The non-symmetric branch is monotone increasing for any $\theta\in[\vartheta(p,d),1]$.\\
\emph{Scenario~2.} The non-symmetric branch is monotone increasing for any $\theta\in(\vartheta_2(p,d),1]$ but it is not optimal near the bifurcation point if $\theta\in[\vartheta(p,d),\vartheta_2(p,d))$.

This also suggests that no other \emph{scenario} can take place, consistently with our numerical computations. Hence we arrive at the conclusion that \emph{Scenario~1} takes place when $\vartheta_2(p,d)<\vartheta(p,d)$ and \emph{Scenario~2} holds if $\vartheta_2(p,d)>\vartheta(p,d)$. The branches of solutions that we have computed are likely to be optimal for the Caffarelli-Kohn-Nirenberg inequalities.

\appendix\section{Some useful quantities}

\subsection{Computing integrals}\label{A1}

We recall that $f(q):=\int_{\R}\frac{ds}{(\cosh s)^q}$ can be explicitly computed: $f(q)=\frac{\sqrt\pi\;\Gamma(\frac q2)}{\Gamma(\frac{q+1}2)}$. An integration by parts shows that $f(q+2)=\frac q{q+1}\,f(q)$. The following formulae are reproduced with no change from~\cite{DDFT} (also see~\cite{DEL2011}). As in~\cite{DDFT}, with $w(s)=(\cosh s)^{-\frac2{p-2}}$, we can define
\[
\mathsf I_q:=\int_{\R}|w(s)|^q\,ds\quad\mbox{and}\quad \mathsf J_2:=\int_{\R}|w'(s)|^2\;ds\,.
\]
Using the function $f$, we can compute $\mathsf I_2=f\big(\frac 4{p-2}\big)$, $\mathsf I_p=f\big(\frac{2\,p}{p-2}\big)=f\big(\frac 4{p-2}+2\big)$ and get the relations
\[
\mathsf I_2=\frac{\sqrt{\pi}\;\Gamma\big(\frac{2}{p-2}\big)}{\Gamma\big(\frac{p+2}{2\,(p-2)}\big)}\;,\quad \mathsf I_p=\frac{4\,\mathsf I_2}{p+2}\;,\quad \mathsf J_2=\frac{4\,\mathsf I_2}{(p+2)\,(p-2)}\,.
\]
As a special case, we have
\[
\mathsf I_p:=\inR{(\cosh s)^{-\frac{2\,p}{p-2}}}=\frac{4\,\sqrt\pi\,\Gamma\(\frac2{p-2}\)}{(p+2)\,\Gamma\(\frac{p+2}{2\,(p-2)}\)}\,.
\]

\subsection{Symmetric extremals and linearization}\label{A2}

Consider $w(s)=(\cosh s)^{-\frac 2{p-2}}$, which is the unique positive solution of
\[
-(p-2)^2\,w''+4\,w-2\,p\,w^{p-1}=0
\]
on $\R$, up to translations. The function $u(s):=\alpha\,w(\beta\,s)$ solves
\be{Eqn:AfterScaling}
-u''+\frac{4\,\beta^2}{(p-2)^2}\,u-\frac{2\,p\,\beta^2}{(p-2)^2}\,\alpha^{2-p}\,u^{p-1}=0\,.
\ee
With $\beta=\frac{p-2}2\,\sqrt\mu$ and $\alpha=(\frac p2\,\mu)^\frac1{p-2}$, $u=u_{\mu,*}$ is given by
\[
u_{\mu,*}(s)=\(\frac p2\,\mu\)^\frac1{p-2}\left[\cosh\(\frac{p-2}2\,\sqrt\mu\,s\)\right]^{-\frac 2{p-2}}\quad\forall\,s\in\R
\]
and solves~Ê\eqref{eq-ustar}.

Next we are interested in computing the ground state energy of the P\"oschl-Teller operator $\mathcal H_\mu=-\frac{d^2}{ds^2}+d-1+\mu-(p-1)\,u_{\mu,*}^{p-2}$, that is the lowest eigenvalue $\lambda_1(\mu)$ in the eigenvalue problem
\[\label{Eqn:Linearized}
\mathcal H_\mu\,\varphi_1=\lambda_1(\mu)\,\varphi_1\,.
\]
See \cite{Landau-Lifschitz-67,poschl1933bemerkungen} for further references. The function
\be{Eqn:varephi1}
\varphi_1(s):=\alpha^\frac p2\,(\cosh(\beta\,s))^{-\frac p{p-2}}=u_{\mu,*}^{p/2} \ee
solves
\[
-\varphi_1''+\frac 14\,\mu\,p^2\,\varphi_1-(p-1)\,u_{\mu,*}^{p-2}\, \varphi_1=0
\]
and provides a solution with
\[
\lambda_1(\mu)=d-1+\mu-\frac 14\,\mu\,p^2\,.
\]
The Sturm-Liouville theory guarantees that $\varphi_1$ generates the ground state and
\[
\inf_{\begin{array}{c}\varphi\in\H^1(\R^d)\cr\nrmC\varphi2^2=\nrmC{u_{\mu,*}}p^p\end{array}}\frac{q[\mu,\varphi]}{\nrmC{u_{\mu,*}}p^p}=\lambda_1(\mu)=\frac{\inC{\varphi_1\,\mathcal H_\mu\,\varphi_1}}{\nrmC\varphi2^2}\,.
\]
Notice that the condition $\lambda_1(\mu_{\rm FS})=0$ determines
\[
\mu_{\rm FS}=4\,\frac{d-1}{p^2-4}\,.
\]
Moreover, this shows that
\[
\lambda_1(\mu)=-\frac{p^2-4}4\,(\mu-\mu_{\rm FS})\,.
\]
Other eigenvalues of $\mathcal H_\mu$ can also be computed using classical transformations and special functions: see~\cite[p.~74]{Landau-Lifschitz-67}. Notice that in Section~\ref{Sec:Bifurcation}, we normalize the function $\varphi=\varphi_1\,f_1$ in the expansion \eqref{Eqn:Expansion} by the condition $\nrmC\varphi2=\nrmC{u_{\mu,*}}p$ consistently with~\eqref{Eqn:varephi1}.

\subsection{Useful quantities}\label{A3}

Collecting results of Sections~\ref{A1} and~\ref{A2}, with $\alpha^{p-2}=\frac p2\,\mu$, we find that
\begin{eqnarray*}
&&\frac{\inC{\varphi^2}}{\inC{u_{\mu,*}^p}}=1\,,\\
&&\frac{\inC{u_{\mu,*}^{p-2}\,\varphi^2}}{\inC{u_{\mu,*}^p}}=\frac{f(q+2)}{f(q)}_{\big|q=\frac{2\,p}{p-2}}\,\alpha^{p-2}=\frac{p^2\,\mu}{3\,p-2}\,,\\
&&\frac{\inC{u_{\mu,*}^{p-4}\,\varphi^4}}{\inC{u_{\mu,*}^p}}=\frac{f(q+4)}{f(q)}_{\big|q=\frac{2\,p}{p-2}}\,\alpha^{2\,(p-2)}=\frac{2\,p^3\,(p-1)\,\mu^2}{(3\,p-2)\,(5\,p-6)}\,,
\end{eqnarray*}
and \begin{eqnarray*}
\frac{\inC{|\nabla\varphi|^2}}{\inC{u_{\mu,*}^2}}&=&\alpha^{p-2}\,\frac{\mathsf I_p}{\mathsf I_2}\left[(d-1)+\(\frac p{p-2}\)^2\,\beta^2\,\(1-\frac{f(q+2)}{f(q)}\)_{\big|q=\frac{2\,p}{p-2}}\right]\\
&=&\frac 12\,\frac{p\,\mu}{p+2}\left[4\,(d-1)+\frac{p^2\,(p-2)\,\mu}{3\,p-2}\right]\,.
\end{eqnarray*}
If $\mu=\mu_{\rm FS}$, then we find that $\frac{\inC{|\nabla\varphi|^2}}{\inC{u_{\mu,*}^2}}=\frac{2\,p\,(p-2)\,(p^2+p-1)}{(p+2)\,(3\,p-2)}\,\mu_{\rm FS}^2$.

\subsection{First spherical harmonics}\label{A4}

Denote by $\zeta\in[0,\pi]$ the azimuthal angle and consider the Laplace-Beltrami operator $\mathcal L$ on the sphere $\S$. When $\mathcal L$ is restricted to functions on $\S$ depending only on $\zeta$, it takes the form
\[
\mathcal L\,f=\sin^{2-d}\zeta\,\frac d{d\zeta}\(\sin^{d-2}\zeta\,\frac{df}{d\zeta}\)
\]
and $\mathcal L$ is unitarily equivalent to $\mathsf L$ defined by
\[
\mathsf L\,g=(1-x^2)\,g''-(d-1)\,x\,g'\quad\;x\in[-1,1]
\]
whose eigenvalues are the \emph{Gegenbauer polynomials} or \emph{ultra-spherical polynomials.} The correspondence between the operators is simply given by
\[
f(\zeta)=g(\cos\zeta)
\]
and one can check that
\[
\int_0^\pi|f(\zeta)|^2\,\sin^{d-2}\zeta\;d\zeta=\int_{-1}^1|g(x)|^2\,(1-x^2)^{\frac d2-1}\,dx\,.
\]
It is also not difficult to check that the first Gegenbauer polynomials are
\[
g_0(x)=1\;,\quad g_1(x)=x\;,\quad g_2(x)=d\,x^2-1\;,\quad g_3(x)=(d+2)\,x^3-3\,x\,,
\]
with eigenvalues respectively equal to $\Lambda_0=0$, $\Lambda_1=d-1$, $\Lambda_2=2\,d$ and $\Lambda_3=3\,(d+1)$:
\[
-\,\mathsf L\,g_0=0\,,\;-\,\mathsf L\,g_1=(d-1)\,g_1\,,\;-\,\mathsf L\,g_2=2\,d\,g_2\,,\;-\,\mathsf L\,g_3=3\,(d+1)\,g_3\,.
\]

On $[0,\pi]$, we consider the probability measure
\[
d\nu(\zeta)=\frac 1{Z_d}\,\sin^{d-2}\zeta\;d\zeta
\]
where $Z_d=\int_0^\pi\sin^{d-2}\zeta\;d\zeta=\sqrt\pi\,\frac{\Gamma\(\frac{d-1}2\)}{\Gamma\(\frac d2\)}$. Then
\[
f_0(\zeta)=1\,,\; f_1(\zeta)=\sqrt d\cos\zeta\,,\; f_2(\zeta)=\sqrt{\frac{d+2}{2\,(d-1)}}\,(d\,\cos^2\zeta-1)
\]
are normalized eigenfunctions in $L^2([0,\pi],d\nu(\zeta))$:
\[
\int_0^\pi|f_i|^2\,d\nu=1\quad\forall\,i=0\,,\;1\,,\;2\,,
\]
with eigenvalues $\Lambda_0=0$, $\Lambda_1=d-1$ and $\Lambda_2=2\,d$:
\[
-\,\mathcal L\,f_0=0\,,\;-\,\mathcal L\,f_1=(d-1)\,f_1\,,\;-\,\mathcal L\,f_2=2\,d\,f_2\,.
\]
We also have the useful formulae
\begin{eqnarray*}
&&\int_0^\pi|f_1|^4\,d\nu=\frac{3\,d}{d+2}\,,\quad\kappa_{(d)}:=\int_0^\pi|f_1|^2\,f_2\,d\nu=\sqrt{\frac{2\,(d-1)}{d+2}}\,,\\
&&f_1^2=f_0+\sqrt{\frac{2\,(d-1)}{d+2}}\,f_2=f_0+\kappa_{(d)}\,f_2\,.
\end{eqnarray*}

\ack{This work has been supported by the ANR project NoNAP. J.D.~also acknowledges support from the ANR project STAB.} We thank Fr\'ed\'eric Hecht and Olivier Pironneau for their help for the implementation of our numerical scheme.\\[4pt]
{\sl\small \copyright~2013 by the authors. This paper may be reproduced, in its entirety, for non-commercial purposes.}

\section*{References}

\end{document}